\documentclass[11pt]{amsart}

\usepackage{amsmath}
\usepackage{amsfonts}
\usepackage{amssymb}
\usepackage{graphicx}
\usepackage{hyperref}
\usepackage{mathrsfs}
\usepackage{pb-diagram}
\usepackage{epstopdf}

\newtheorem{theorem}{Theorem}[section]
\newtheorem{lemma}[theorem]{Lemma}
\newtheorem{corollary}[theorem]{Corollary}
\newtheorem{prop}[theorem]{Proposition}
\newtheorem{defn}[theorem]{Definition}

\newtheorem{remark}[theorem]{Remark}

\numberwithin{equation}{section}

\newcommand{\dd}[1]{\frac{\partial}{\partial #1}}

\newcommand{\pairing}[2]{\left( #1\, , \, #2 \right)}

\newcommand{\nat}{\mathbb{N}}
\newcommand{\integer}{\mathbb{Z}}
\newcommand{\rat}{\mathbb{Q}}
\newcommand{\real}{\mathbb{R}}
\newcommand{\cpx}{\mathbb{C}}

\newcommand{\consti}{\mathbf{i}\,}
\newcommand{\conste}{\mathbf{e}}

\newcommand{\proj}{\mathbb{P}}
\newcommand{\torus}{\mathbf{T}}
\newcommand{\moduli}{\mathcal{M}}
\newcommand{\point}{\mathrm{pt}}

\newcommand{\rank}{\mathrm{rank}}

\newcommand{\Jac}{\mathrm{Jac}}

\begin{document}

\title[Open GW and mirror map]{Open Gromov-Witten invariants and mirror maps for semi-Fano toric manifolds}

\author[K. Chan]{Kwokwai Chan}
\address{Department of Mathematics\\ The Chinese University of Hong Kong\\ Shatin \\ Hong Kong}
\email{kwchan@math.cuhk.edu.hk}
\author[S.-C. Lau]{Siu-Cheong Lau}
\address{Department of Mathematics and Statistics\\Boston University\\Boston, MA\\USA}
\email{lau@math.bu.edu}
\author[N.C. Leung]{Naichung Conan Leung}
\address{The Institute of Mathematical Sciences and Department of Mathematics\\ The Chinese University of Hong Kong\\ Shatin \\ Hong Kong}
\email{leung@math.cuhk.edu.hk}
\author[H.-H. Tseng]{Hsian-Hua Tseng}
\address{Department of Mathematics\\ Ohio State University\\ 100 Math Tower, 231 West 18th Ave. \\ Columbus \\ OH 43210\\ USA}
\email{hhtseng@math.ohio-state.edu}

\date{\today}

\dedicatory{Dedicated to Prof. Kyoji Saito on the occasion of his 75th birthday.}

\begin{abstract}
We prove that for a compact toric manifold whose anti-canonical divisor is numerically effective, the Lagrangian Floer superpotential defined by Fukaya-Oh-Ohto-Ono \cite{FOOO1} is equal to the superpotential written down by using the toric mirror map under a convergence assumption. This gives a method to compute open Gromov-Witten invariants using mirror symmetry.
\end{abstract}

\maketitle


\section{Introduction} \label{intro}

The PDE approach to mirror symmetry of toric manifolds is a very well-developed subject, see for example \cite{givental98} and \cite{LLY3}. More recently, the Lagrangian Floer theory developed by Fukaya-Oh-Ohta-Ono \cite{FOOO_I, FOOO_II, FOOO1, FOOO2} gives a geometric approach to study mirror symmetry for toric manifolds.  The purpose of this paper is to relate these two seemingly different approaches in the case of compact toric manifolds.  By doing so, we obtain an open analogue of the closed-string mirror symmetry discovered by Candelas-de la Ossa-Green-Parkes \cite{candelas91}.  Namely, under mirror symmetry, the computation of open Gromov-Witten (GW) invariants is transformed into a PDE problem of solving Picard-Fuchs equations.

\subsection{Superpotentials}
Let $X$ be a compact toric manifold of complex dimension $n$ and $q$ a K\"ahler class of $X$.  The mirror of $(X,q)$ is a Landau-Ginzburg model $W_q$, which is a holomorphic function on $(\cpx^\times)^n$.  Closed-string mirror symmetry states that the deformation of $W_q$ encodes closed GW invariants of $X$.  More precisely, there is an isomorphism
$$ QH^* (X,q) \cong \Jac (W_q)$$
as Frobenius algebras, where $QH^* (X,q)$ denotes the small quantum cohomology ring of $(X,q)$ and
$$ \Jac (W_q) := \frac{\cpx[z_1^{\pm 1}, \ldots, z_n^{\pm 1}]}{\left\langle z_1 \frac{\partial W_q}{\partial z_1}, \ldots, z_n \frac{\partial W_q}{\partial z_n} \right\rangle}$$
is the Jacobian ring of $W_q$.

Based on physical arguments, Hori-Vafa \cite{hori00} gave a recipe to write down a Laurent polynomial $W_q^{\circ}$ from the combinatorial data of $X$.  Independently, $W_q^\circ$ was also constructed by Givental \cite{givental_ICM}. It turns out that $W_q^{\circ}$ gives the `leading order term' of $W_q$, and our results show that the remaining terms are {\em instanton corrections} coming from holomorphic disks.

The PDE approach to writing down these instanton corrections is achieved by solving a Picard-Fuchs system for the {\em mirror map} $\check{q} (q)$. \footnote{In the literatures the mirror map refers to $q(\check{q})$, while $\check{q} (q)$ is its inverse.}  It was studied by Givental \cite{givental98} and Lian-Liu-Yau \cite{LLY3} for a toric manifold $X$ whose anti-canonical line bundle $-K_X$ is numerically effective.  We call such $X$ a {\em semi-Fano} toric manifold.  In this setting the instanton-corrected superpotential is then given by
$$ W^{\text{PF}}_q := W_{\check{q} (q)}^{\circ}.$$
The function $W^{\text{PF}}_q$ fits into the mirror symmetry picture mentioned above, namely we have
$$ QH^* (X,q) \cong \Jac (W^{\text{PF}}_q)$$
as Frobenius algebras.

On the other hand, the instanton corrections are realised by Fukaya-Oh-Ohta-Ono \cite{FOOO1} using open GW invariants as follows.  Let $\torus \subset X$ be a Lagrangian toric fiber and $\pi_2 (X,\torus)$ the set of homotopy classes of maps $(\Delta,\partial \Delta) \to (X,\torus)$, where $\Delta$ denotes the closed unit disk.  For $\beta \in \pi_2 (X,\torus)$, the moduli space $\moduli_1(\beta)$ of stable disks representing $\beta$ and its virtual fundamental class $[\moduli_1(\beta)] \in H_n(\torus)$ are defined.  The one-pointed open GW invariant associated to $\beta$ is defined as
$$ n_\beta := \int_{[\moduli_1(\beta)]} \text{ev}^*[\point] $$
where $[\point] \in H^n(\torus)$ is the point class and $\text{ev}:\moduli_1(\beta)\to\torus$ is the evaluation map. Then the instanton-corrected superpotential is defined as
$$ W^{\text{LF}}_q := \sum_{\beta \in \pi_2 (X,\torus)} n_\beta Z_{\beta} $$
where $Z_{\beta}$ is an explicitly written monomial for each $\beta$ (see Section \ref{LG} for more details).  Notice that the above formal sum involves infinitely many terms in general, and  is well-defined over the Novikov ring (see Section \ref{FOOO}).  Fukaya-Oh-Ohta-Ono \cite{FOOO10b} proved that
$$ QH^* (X,q) \cong \Jac (W^{\text{LF}}_q)$$
as Frobenius algebras.

While $W^{\text{PF}}$ and $W^{\text{LF}}$ originate from totally different approaches, they lead to the same mirror symmetry statements.  In view of this, it is natural to expect that
\begin{equation} \label{PF=LF}
W^{\text{PF}} = W^{\text{LF}}.
\end{equation}

When $X$ is a Fano toric manifold, i.e. $-K_X$ is ample, then the toric mirror map $\check{q} (q)$ for $X$ is trivial, as observed in \cite{givental98}. Hence $W^{\text{PF}}=W^{\circ}$. Also, for a Fano toric manifold $X$, the open GW invariants $n_\beta$ considered above can be completely calculated using the work \cite{cho06}. It follows that $W^{\text{LF}}=W^\circ$. Therefore the formula \eqref{PF=LF} holds true for Fano toric manifolds.

When $X$ is semi-Fano but not Fano, i.e. $-K_X$ is numerically effective but not ample, the toric mirror map $\check{q} (q)$ for $X$ is non-trivial, and the open GW invariants $n_\beta$ cannot be easily calculated due to nontrivial obstructions to the moduli problem. In this situation, the formula \eqref{PF=LF} is a highly non-trivial statement.

In this paper we prove \eqref{PF=LF} under the technical assumption that the coefficients of $W^{\text{LF}}$ converge analytically (instead of just being formal sums):

\begin{theorem}[Restatement of Theorem \ref{main_thm}]\label{main_thm_intr}
Let $X$ be a toric manifold with $-K_X$ numerically effective, and let $W^{\text{PF}}$ and $W^{\text{LF}}$ be the superpotentials in the mirror as explained above.  Then
$$ W^{\text{PF}} = W^{\text{LF}} $$
provided that each coefficient of $W^{\text{LF}}$ converges in an open neighborhood around $q=0$.
\end{theorem}

Our approach to Theorem \ref{main_thm_intr} is analytic in nature, involving the theory of unfoldings of analytic functions. The proof may be be summarized as follows. By combining the two isomorphisms $QH^* (X,q) \cong \Jac (W^{\text{PF}}_q)$ and $QH^* (X,q) \cong \Jac (W^{\text{LF}}_q)$, we get an isomorphism $\Jac (W^{\text{LF}}_q) \cong \Jac (W^{\text{PF}}_q)$.  This isomorphism together with semi-simplicity of the Jacobian rings imply that $W^{\text{LF}}_q$ and $W^{\text{PF}}_q$ have the same critical values.  Putting these functions into a universal unfolding gives us constant families of critical points linking those of $W^{\text{PF}}_q$ and $W^{\text{LF}}_q$.  Since they have the same critical values, they indeed correspond to the same based point in the universal family. It follows that the two functions coincide. Details are given in Section \ref{FOOO=G}.

In this paper we show that the technical convergence assumption in Theorem \ref{main_thm_intr} holds at least in the following cases: (1) when $\dim X = 2$ (see Section \ref{surface}) and (2) when $X$ is of the form $\proj (K_S \oplus \mathcal{O}_S)$ for some toric Fano manifold $S$ (see Section \ref{can_line}). 

\begin{remark}
The isomorphism $QH^* (X,q) \cong \Jac (W^{\text{PF}}_q)$ in closed string mirror symmetry for toric manifolds is in fact a consequence of a more involved correspondence. Following Givental \cite{givental_ICM}, the quantum cohomology ring $QH^*(X,q)$ can be equipped with a $D$-module structure which is often called the {\em quantum $D$-module}. By K. Saito's theory of primitive forms \cite{K_Saito}, the superpotential $W^{\text{PF}}$ also defines a $D$-module. Mirror symmetry for the toric manifold $X$ may be understood as an isomorphism between the quantum $D$-module and the $D$-module defined by $W^{\text{PF}}$.

Nevertheless, our proof of \eqref{PF=LF} does not require the mirror symmetry results at the level of $D$-module. The isomorphisms between quantum cohomology rings and Jacobian rings as Frobenius algebras are sufficient.
\end{remark}

\subsection{Computation of open GW invariants}
The function $W^{\text{LF}}$ is a generating function of open GW invariants $n_\beta$ (and thus can be regarded as an object in the `A-side'), whereas $W^{\text{PF}}$ arises from solving Picard-Fuchs equations (and so is an object in the `B-side'). Using this equality, the task of computing the open GW invariants is transformed to solving Picard-Fuchs equations which has been known to experts.  Thus our work gives a mirror symmetry method to compute open GW invariants.

More precisely, in Section \ref{sec:computing_openGW} we derive from Theorem \ref{main_thm} an explicit formula for the generating functions of open GW invariants $n_\beta$. Our formula involves explicit hypergeometric series and mirror maps, and can be used to effectively evaluated open GW invariants $n_\beta$ term-by-term. As an application of our formula, we calculate some open GW invariants in a non-trivial example in Section \ref{sec:further_example}.

Our formula for open GW invariants are made even more explicit using a relationship with Seidel representations. In Section \ref{Seidel_rep}, we derive from \eqref{PF=LF} the following formula for the generating function $\delta_i(q)$ of open GW invariants: under the toric mirror map $\check{q}=\check{q}(q)$, we have
\begin{equation}\label{opGW=g_0:intro}
1+\delta_i(q)=\exp\left(g_0^{(i)}(\check{q})\right).
\end{equation}
Here $g_0^{(i)}(\check{q})$ is the following power series\footnote{It is easy to see that this power series is convergent.}
\begin{equation*}
g_0^{(i)}(\check{q}) := \sum_{\substack{\pairing{-K_X}{d} = 0 \\ \pairing{D_i}{d} < 0 \\ \pairing{D_j}{d} \geq 0 \,\, \forall j\neq i }} \frac{(-1)^{\pairing{D_i}{d}}(-\pairing{D_i}{d}-1)!}{\prod_{j\not=i}\pairing{D_j}{d}!} \check{q}^d.
\end{equation*}

Our formula \eqref{opGW=g_0:intro} completely and effectively calculates open GW invariants of {\em all} semi-Fano torc manifolds. This is a significant advance, since prior to our work open GW invariants of toric manifolds are only calculated in a few examples.

Our main results can also be understood as providing a geometric interpretation of the toric mirror maps $\check{q} (q)$, whose definition is combinatorial in nature and somewhat mysterious.

\subsection{Outline}
This paper is organized as follows.  Sections \ref{toric-LG} and \ref{unfold} serve as short reviews on toric geometry and deformation theory of analytic functions respectively.  Section \ref{FOOO=G} contains the proof of the main theorem and its applications to computation of open GW invariants. Inspired by the recent work of Gonz\'alez and Iritani \cite{G-I11} on the relation between mirror maps and Seidel representations \cite{seidel97}, \cite{McDuff_seidel}, we explain in Section \ref{Seidel_rep} how \eqref{PF=LF} implies that open GW invariants can also be expressed by using Seidel representations.

Here are some remarks on notations.  $H_2(X)$ means $H_2(X, \integer)$ unless otherwise specified.  $QH^*(X)$ always denotes the small quantum cohomology of $X$.  The work of Fukaya-Oh-Ohta-Ono will be abbreviated as `FOOO' in this paper.

\subsection{Some history of this paper}
The first draft of this paper was finished in October 2011. The technical assumption in Theorem \ref{main_thm_intr}, namely, the assumption that coefficients of $W^{\text{LF}}$ converge in an open neighborhood around $q=0$, is required in order to apply the theory of unfoldings of analytic functions. It may be possible to avoid this assumption by carrying out the arguments for functions taking values in suitable formal power series rings.
Around a year later (in November 2012), we found a new geometric approach making use of Seidel representations that proves \eqref{PF=LF} {\em unconditionally}; this geometric proof, together with several applications (one of which being, in turn, the convergence of coefficients of $W^{\text{LF}}$!), appeared in \cite{CLLT12}. Nevertheless we believe that it is still valuable to retain the argument in the current paper.

\section*{Acknowledgment}
We are grateful to Hiroshi Iritani for useful discussions. Part of this work was done when K. C. was working as a project researcher at IPMU at the University of Tokyo and visiting IH\'ES. He would like to thank both institutes for hospitality and providing an excellent research environment. S.-C. L. expresses his deep gratitude to Kyoji Saito for interesting discussions at IPMU on Frobenius structures and primitive forms, and also to Kaoru Ono for his hospitality at Hokkaido University and interesting discussions on Lagrangian Floer theory.

The work of K. C. was supported in part by a grant from the Hong Kong Research Grants Council (Project No. CUHK404412).  The work of S.-C. L. was supported by IPMU and Harvard University.  The work of N. C. L. was supported by a grant from the Hong Kong Research Grants Council (Project No. CUHK401809). The work of H.-H. T. was supported in part by NSF grant DMS-1047777.

\section{Toric manifolds and their Landau-Ginzburg mirrors} \label{toric-LG}

In this section we give a quick review on some facts on toric manifolds.  Then we recall the mirror maps for toric manifolds and Lagrangian Floer theory
which we will use in this paper.  The toric mirror map, which arises from attempts to compute genus $0$ GW invariants of toric manifolds, has been studied by Givental \cite{givental98} and Lian-Liu-Yau \cite{LLY3}, while the Lagrangian Floer theory was constructed by Fukaya-Oh-Ohta-Ono \cite{FOOO1}.

\subsection{A quick review on toric manifolds} \label{toric}

Let $N\cong\integer^n$ be a lattice of rank $n$. For simplicity we shall always use the notation $N_R:=N\otimes R$ for a $\integer$-module $R$. Let $X = X_\Sigma$ be a compact complex toric $n$-fold defined by a fan $\Sigma$ supported in $N_\real$. $X_\Sigma$ admits an action by the complex torus $N_\cpx/N\cong(\cpx^\times)^n$, whence its name `toric manifold'. There is an open orbit in $X$ on which $N_\cpx/N$ acts freely, and by abuse of notation we shall also denote this orbit by $N_\cpx/N\subset X$.  Roughly speaking, $X$ is obtained from the open part $N_\cpx/N$ by compactifying along every ray of $\Sigma$.

We denote by $M=\text{Hom}(N,\integer)$ the dual lattice of $N$. Every lattice point $\nu \in M$ gives a nowhere-zero holomorphic function
$\exp\pairing{\nu}{\cdot}:N_\cpx/N\to\cpx$ which extends to a meromorphic function on $X_\Sigma$. Its zero and pole set gives a toric divisor\footnote{A divisor $D$ in $X$ is toric if $D$ is invariant under the action of $N_\cpx/N$ on $X$.} which is linearly equivalent to $0$.

If we further equip $X$ with a toric K\"ahler form $\omega \in \Omega^2 (X, \real)$, then the action of $N_\real/N$ on $X_\Sigma$ induces a moment map
$$\mu_0:X\to M_\real,$$
whose image is a polytope $P\subset M_\real$ defined by a system of inequalities
\begin{equation} \label{poly_ineq}
\pairing{v_i}{\cdot}\geq c_i,\ i=1,\ldots,m
\end{equation}
where $v_i$ are all primitive generators of rays of $\Sigma$, and $c_i\in\real$ are some suitable constants.  Since translation of the polytope $P$ does not affect the K\"ahler class, without loss of generality we may assume $c_1 = \ldots = c_n = 0$.

We always denote a regular moment map fiber of $\mu_0$ over $r\in M_\real$ by $\torus_r$, and sometimes the subscript $r$ is omitted if the base point is not important for the discussion.  The primitive generators $v_i$'s correspond to disk classes $\beta_i(r) \in \pi_2(X,\torus_r)$, which are referred as the basic disk classes.  The symplectic areas of these disk classes are given by (see \cite{cho06})
$$ \int_{\beta_i(r)} \omega = 2\pi\big(\pairing{v_i}{r} - c_i\big). $$

To complexify the K\"ahler moduli so that it is comparable to the mirror complex moduli, one considers complexified K\"ahler forms $\omega_\cpx = \omega + \consti B \in \Omega^2 (X, \cpx)$ where $B$ is any closed real two form.  One obtains the complexified K\"ahler cone $\mathcal{M}_{A}(X) \subset H^2(X,\cpx)$ by collecting the classes of all such complexified K\"ahler forms.  Let $\{\mathbf{p}_1, \ldots, \mathbf{p}_l\}$ be a nef basis of $H^2(X)$, and let $\{\Psi_1, \ldots, \Psi_l\} \subset H_2(X)$ be its dual basis.  Then $\Psi_k$ induce coordinate functions $q_k$ on $H^2(X,\cpx)$ by assigning $\eta \in H^2(X,\cpx)$ with the value $\conste^{-\pairing{q_k}{\eta}}$.  In particular, we may restrict them on $\mathcal{M}_{A}(X)$ to get coordinates for the K\"ahler moduli.  Notice that $(q_1, \ldots, q_l)$ tends to $0$ when one takes the large radius limit.

Here comes a notational convention: For a class $d \in H_2(X)$, define
$$ q^d = \prod_{j=1}^l q_j^{\pairing{\mathbf{p}_j}{d}}.$$
In this expression we may regard $q = (q_1, \ldots, q_l)$ simply as formal parameters, not necessarily as coordinates of the K\"ahler moduli of $X$.  If $d$ is a curve class, since $\mathbf{p}_j$ is nef for all $j=1, \ldots, l$, the exponents of $q_j$ in the above product are all nonnegative.  This fact is important when one considers the $J$-function, which is a formal power series in $q$ and it lives in the Novikov ring due to this fact.

The polytope $P$ admits a natural stratification by its faces.  Each codimension-one face $T_i\subset P$ which is normal to $v_i\in N$ gives an irreducible toric divisor $D_i=\mu_0^{-1}(T_i)\subset X_\Sigma$ for $i=1,\ldots,m$, and all other toric divisors are generated by $\{D_i\}_{i=1}^m$.  For example, the anti-canonical divisor $-K_X$ of $X$ is given by $\sum_{i=1}^m D_i$.

\subsection{The Landau-Ginzburg mirrors of toric manifolds} \label{LG}

The mirror of a toric manifold $X = X_{\Sigma}$ is a Landau-Ginzburg model $(\check{X},W)$, where $\check{X} = M_\cpx / M \cong (\cpx^\times) ^n$ and $W:\check{X}\to\cpx$ is a holomorphic function called the superpotential.
This subsection reviews how to use the combinatorial data of $\Sigma$ to write down the superpotential.\footnote{The superpotential appearing in this section has not received instanton corrections yet.  In the next two subsections we review two approaches to correct the superpotential, which are provided by mirror maps and Lagrangian Floer theory respectively.}  It is commonly called the Hori-Vafa superpotential \cite{hori00} in the literatures, but in fact it has appeared earlier in Givental's paper \cite{givental98} (the notation for the superpotential in Givental's paper being $F_0(u)$).

Recall that for $i = 1, \ldots, m$ and a moment map fiber $\torus_r = \mu_0^{-1}(r)$ at $r \in P^{\circ}$, $\beta_i \in \pi_2(X,\torus_r)$ denotes the basic disk class bounded by $\torus_r$ corresponding to the primitive generator $v_i$ of a ray of $\Sigma$.  We may also write it as $\beta_i(r)$ to make the dependency on $r$ more explicit.  One has the following exact sequence
\begin{equation} \label{pi2pi1}
0 \to \pi_2(X) \to \pi_2(X,\torus_r) \to \pi_1(\torus_r) \to 0
\end{equation}
where the first map $\pi_2(X) \to \pi_2(X,\torus_r)$ is a natural inclusion and the second map $\partial: \pi_2(X,\torus_r) \to \pi_1(\torus_r)$ is given by taking boundary.  Moreover $\pi_1(\torus_r)$ is canonically identified with the lattice $N$, and $\pi_2(X) = H_2(X)$ since $\pi_1(X) = 0$.

\begin{defn}[The Hori-Vafa superpotential] \label{defn_W_HV}
Let $X$ be a toric manifold equipped with a toric K\"ahler form $\omega$, and let $P$ be the corresponding moment map polytope.  Denote the primitive generators of rays in its fan by $v_i$, $i=1,\ldots,m$, and the corresponding basic disk classes bounded by a regular moment map fiber over $r \in P^\circ$ by $\beta_i(r)$ (here $P^\circ$ denotes the interior of $P$).  The Hori-Vafa superpotential mirror to $X$ is defined to be
\begin{align*}
W^{\circ}: P^{\circ} \times M_\real / M \to \cpx,\\
W^{\circ}(r,\theta) = \sum_{i=1}^m Z_i(r,\theta)
\end{align*}
in which the summands are
\begin{equation} \label{Z_i}
Z_i(r,\theta) = \exp\left(-\int_{\beta_i(r)} \omega + 2\pi\consti \pairing{v_i}{\theta} \right).
\end{equation}
\end{defn}

A small circle is placed as the superscript to indicate that this superpotential has not received instanton corrections yet.

To make the above expression more explicit, let us fix a top dimensional cone of $\Sigma$ generated by, say, $v_1, \ldots, v_n \in N$ (we can always assume that the cone is generated by $v_1, \ldots, v_n$ by relabeling the primitive generators if necessary).  Then each $v_i$ defines a coordinate function
$$z_i := \exp \left(2\pi\consti \pairing{v_i}{\cdot}\right): M_\cpx/ M \to \cpx,$$
for $i = 1, \ldots, n$.  One may write down the Hori-Vafa superpotential in terms of these coordinates $z_i$ as follows:

\begin{prop} \label{write W_HV}
Assume the same setting as in Definition \ref{defn_W_HV}.  The Hori-Vafa superpotential can be written as
\begin{equation} \label{W_HV}
W^{\circ}_q = z_1 + \ldots + z_n + \sum_{i=n+1}^{m} q^{\alpha_i} z^{v_i}
\end{equation}
where $z^{v_i} := \prod_{k=1}^n z_k^{v_i^k}$ and
\begin{equation} \label{alpha}
\alpha_i := \beta_i - \sum_{k=1}^{n} v_i^k \beta_k
\end{equation}
are classes in $H_2(X)$ for $i = n+1, \ldots, m$.
\end{prop}

\begin{proof}
Since $\int_{\beta_i(r)} \omega = 2 \pi (\pairing{v_i}{r} - c_i)$, we have
$$ Z_i(r,\theta) = \conste^{-2\pi\consti c_i} \exp \left(2\pi \consti \pairing{v_i}{\theta + \consti r}\right).$$
In Section \ref{toric} we have made the choice $c_i = 0$ for $i = 1, \ldots n$.  Thus $Z_i = z_i$ for $i = 1, \ldots, n$.

For $i=n+1, \ldots, m$, we may write
$$ v_i = \sum_{k=1}^{n} v_i^k v_k $$
for $v_i^k \in \integer$, because $\{v_1, \ldots, v_n\}$ generates $N$.  Then
$$ \partial \left( \beta_i - \sum_{k=1}^{n} v_i^k \beta_k \right) = 0.$$
From the exact sequence \eqref{pi2pi1}, $\alpha_i := \beta_i - \sum_{k=1}^{n} v_i^k \beta_k$ belong to $H_2(X)$ for all $i = n+1, \ldots, m$.

Recall that we have introduced a basis $\{\mathbf{p}_1, \ldots, \mathbf{p}_l\}$ of $H_2(X)$ in the previous subsection, and $q^{\alpha_i} = q_1^{\pairing{\mathbf{p}_1}{\alpha_i}} \ldots q_l^{\pairing{\mathbf{p}_l}{\alpha_i}}$.  Then for $i = n+1, \ldots, m$,
\begin{align*}
Z_i(r,\theta) &= \exp\left(-\int_{\beta_i(r)} \omega + 2\pi\consti \pairing{v_i}{\theta} \right) \\
&= \exp\left(-\int_{\alpha_i} \omega - \sum_{k=1}^{n} v_i^k \left(\int_{\beta_k}\omega + 2\pi\consti \pairing{v_k}{\theta}\right) \right)\\
&= q^{\alpha_i}(\omega) z^{v_i}
\end{align*}
Thus the Hori-Vafa superpotential can be written as
$$
W^{\circ} = z_1 + \ldots + z_n + \sum_{i=n+1}^{m} q^{\alpha_i} z^{v_i}.
$$
\end{proof}

Note that the expression of $W^{\circ}_q$ appearing in Proposition \ref{write W_HV} only exploits the fan configuration of the toric manifold $X$ and does not involve its K\"ahler structure.  The Hori-Vafa superptential corresponding to $(X,\omega_\cpx)$, where $\omega_\cpx$ is a complexified K\"ahler class of $X$, is $W^{\circ}_{q(\omega_\cpx)}$ where $q(\omega_\cpx) = \left(\conste^{-\int_{\Psi_1} \omega_\cpx}, \ldots , \conste^{-\int_{\Psi_l} \omega_\cpx}\right)$ is the coordinate of $\omega_\cpx$ in the complexified K\"ahler moduli.  We omit the subscript $q$ in the notation $W^{\circ}_q$ whenever the dependency on $q$ is not relevant for the discussion.

From this expression, we see that $W^{\circ}$ (whose domain is originally $P^{\circ} \times M_\real / M$) can be analytically continued to $M_\cpx / M \cong (\cpx^\times)^n$.  We will mainly be interested in the deformation of $W^{\circ}$, which is captured by its Jacobian ring $\Jac (W^{\circ})$ whose definition is as follows:

\begin{defn}[The Jacobian ring] \label{Jac_def}
Let $f:\mathcal{D} \to \cpx$ be a holomorphic function on a domain $\mathcal{D} \subset \cpx^n$.  Then the Jacobian ring of $f$ is defined as
$$\Jac (f) := \frac{\mathcal{O}_D}{\mathcal{O}_D \langle \partial_1 f, \ldots, \partial_n f \rangle}$$
where $\mathcal{O}_D$ denotes the ring of holomorphic functions on $\mathcal{D}$.
\end{defn}

The Jacobian ring of $W^{\circ}$ is deeply related with the quantum cohomology ring $QH^*(X)$ of $X$.  In \cite{batyrev93}, Batyrev defined the following ring $B_X$ for a toric manifold $X$ and it was later shown by Givental \cite{givental98} that the small quantum cohomology ring $QH^*(X)$ is isomorphic to $B_X$ as algebras when $X$ is Fano:

\begin{defn}[The Batyrev ring \cite{batyrev93}]
Let $X$ be a toric manifold whose toric divisors are denoted by $D_1, \ldots, D_m$, and let $\{\mathbf{p}_j\}_{j=1}^l$ be a nef basis of $H^2(X)$, so that $D_i = \sum_{j=1}^l a_{ij} \mathbf{p}_j$ for $a_{ij} \in \integer$.  The Batyrev ring for $q = (q_1,\ldots,q_l) \in \cpx^l$ is defined as
$$B_X(q) := \cpx[u_1, \ldots, u_l] / \mathcal{I}_{B_X} $$
where $\mathcal{I}_{B_X}$ is the ideal generated by the elements
$$
\prod_{j: \pairing{D_j}{d} > 0} w_j^{\pairing{D_j}{d}} - q^{d} \prod_{j: \pairing{D_j}{d} < 0} w_j^{-\pairing{D_j}{d}}
$$
for every $d \in H_2(X)$.  In the above expression
\begin{equation} \label{w_i}
w_i := \sum_{j=1}^l a_{ij} u_j
\end{equation}
for $i=1,\ldots,m$.
\end{defn}

\begin{theorem}[Batyrev's mirror theorem \cite{batyrev93, givental98}] \label{Bat_mir_thm}
Let $X$ be a toric Fano manifold equipped with a complexified K\"ahler class $\omega_\cpx$.  Denote the small quantum cohomology ring of $X$ by $QH^*(X,\omega_\cpx)$.  Then
$$QH^*(X,\omega_\cpx) \cong B_X(q(\omega_\cpx)) $$
by sending the generators $\mathbf{p}_i \in QH^*(X,\omega_\cpx)$ to $u_i \in B_X(q(\omega_\cpx))$.  Here to define the Batyrev ring we choose a nef basis $\{\mathbf{p}_j\}_{j=1}^l$ of $H^2(X)$ whose dual basis is denoted by $\{\Psi_i\}_{i=1}^l \subset H_2(X)$, and $q(\omega_\cpx) = (\conste^{-\int_{\Psi_1}\omega_\cpx}, \ldots, \conste^{-\int_{\Psi_l}\omega_\cpx})$ is the coordinate of $\omega_\cpx$ in the complexified K\"ahler moduli.
\end{theorem}

On the other hand, the Batyrev ring $B_X$ is known to be isomorphic to the Jacobian ring.  A good reference is part (i) of Proposition 3.10 of \cite{iritani09} by Iritani.

\begin{prop}[\cite{batyrev93,iritani09}] \label{Bat-Jac}
Let $X$ be a compact toric manifold and $W^{\circ}_q$ its Hori-Vafa superpotential.  Then
$$ B_X(q) \cong \Jac (W_q^{\circ}) $$
as algebras, where the isomorphism is given by taking $u_j$ to $q_j\frac{\partial  W_q^{\circ}}{\partial q_j}$ for $j=1,\ldots,l$.
\end{prop}

Combining Theorem \ref{Bat_mir_thm} and Proposition \ref{Bat-Jac}, one has an isomorphism of algebras between $QH^*(X,\omega_\cpx)$ and $\Jac (W^{\circ}_{q(\omega_\cpx)})$ when $X$ is a Fano manifold.  However, this statement no longer holds in general when $X$ is non-Fano.  One needs to include `instanton corrections' to make similar statements for non-Fano toric manifolds.  In the next two sections we will review two different approaches in the semi-Fano setting.

\subsection{Toric mirror transform and mirror theorems} \label{Giv}

Mirror symmetry is powerful because it transforms quantum invariants to some classically known quantities in the mirror side.  In this toric setting, it transforms the small quantum cohomology ring $QH^* (X)$ to the Jacobian ring of the superpotential in the mirror side.  However, to make such a statement the above expression \eqref{W_HV} for $W^\circ$ has to be modified by instanton corrections.  In this section, we review the approach by using the toric mirror transforms studied by \cite{givental98} and \cite{LLY3}.  In the next section we will review another approach which uses open GW invariants \cite{FOOO1}.  The ultimate goal of this paper is to prove that these two approaches are equivalent, and this statement will be made clear in Section \ref{FOOO=G}.

From now on we shall always assume that $X$ is a semi-Fano toric manifold, which means the following:

\begin{defn}
A compact complex manifold is said to be semi-Fano if its anti-canonical divisor $-K_X$ is numerically effective, that is, $-K_X \cdot C \geq 0$ for every complex curve $C$ in $X$.
\end{defn}

Under this condition the toric mirror transform can be written down explicitly. In Givental's formulation \cite{givental98}, this is done by matching the $I$-function with the $J$-function, which are $H^*(X, \cpx)$-valued functions.

Recall that the $I$-function is written as
\begin{equation}
I(\hat{q},z)=z \, \conste^{(\mathbf{p}_1 \log\hat{q}_1 + \ldots + \mathbf{p}_l \log\hat{q}_l) /z} \sum_{d\in H_2^\mathrm{eff}(X)} \hat{q}^{d} \prod_i\frac{\prod_{m=-\infty}^0(D_i+mz)}{\prod_{m=-\infty}^{D_i\cdot d}(D_i+mz)}
\end{equation}
where $\{\mathbf{p}_1, \ldots, \mathbf{p}_l\}$ is a nef basis of $H^2(X)$ chosen in Section \ref{toric}, $\hat{q} = (\hat{q}_1, \ldots, \hat{q}_l)$ are formal variables\footnote{Conceptually $\hat{q}_k$ are coordinates of the formal neighborhood around the large complex structure limit of the mirror complex moduli.  That is, the mirror complex moduli is given by $\text{Spec } \cpx[\hat{q}_1^{\pm 1}, \ldots, \hat{q}_l^{\pm 1}]$, the large complex structure limit is at $\hat{q}_1 = \ldots = \hat{q}_l = 0$, and $I$ is a $H^*(X,\cpx)$-valued function defined on $\text{Spec } \cpx[[\hat{q}_1, \ldots, \hat{q}_l]]$, a formal neighborhood of $0$.}, $d \in H_2(X)$ is written as $d = \sum_{k=1}^l d_k q_k$ and $\hat{q}^d := \prod_{k=1}^l \hat{q}_k^{d_k}$.  Moreover $z$ is a formal variable (caution: it has nothing to do with the coordinates $z_i$ on $M_\cpx / M \cong (\cpx^\times)^n$ given in the last section).  $I$ results from oscillatory integrals of $W^\circ$ and thus captures information about the Landau-Ginzburg mirror.

The $J$-function is a generating function recording the descendent invariants of $X$ as follows:
\begin{equation} \label{J-function}
J(q,z)=z \, \conste^{(\mathbf{p}_1 \log q_1 + \ldots + \mathbf{p}_l \log q_l) /z} \Bigg(1+\sum_\alpha\sum_{d \in H_2^\mathrm{eff}(X)\setminus\{0\}} q^d \Big\langle1,\frac{\phi_\alpha}{z-\psi}\Big\rangle_{0,2,d}\phi^\alpha\Bigg),
\end{equation}
where $\{\phi_\alpha\}$ is a homogeneous additive basis of $H^*(X)$ and $\{\phi^\alpha\} \subset H^*(X)$ is its dual basis with respect to the Poincar\'e pairing.  We always use $\langle \ldots \rangle_{g,k,d}$ to denote the genus $g$, degree $d$ GW invariant of $X$ with $k$ insertions. $\Big\langle1,\frac{\phi_\alpha}{z-\psi}\Big\rangle_{0,2,d}$ is expanded into a power series in $z^{-1}$ whose coefficients are descendent invariants of $X$, which involve the $\psi$-classes in GW theory.

While $I$-function is explicitly written down in terms of combinatorial data of the fan $\Sigma$, $J$-function involves descendent invariants of $X$ which are difficult to compute in general.  It was shown\footnote{The mirror theorem works much more generally for semi-Fano complete intersections in toric varieties; here we only need its restriction to semi-Fano toric cases.} by \cite{givental98, LLY3} that via the `mirror transform' $\hat{q} = \hat{q}(q)$, $J$ can be expressed in terms of $I$:

\begin{theorem}[Toric mirror theorem \cite{givental98, LLY3}] \label{mir_thm_Giv}
Let $X$ be a semi-Fano toric manifold.  There exist formal power series $\hat{q}_i(q)$ for $i = 1, \ldots, l$ such that
$$I(\hat{q}(q)) = J(q)$$
where $\hat{q}(q) = \big(\hat{q}_1(q), \ldots, \hat{q}_l(q)\big)$. Moreover the power series $\hat{q}_i(q), i = 1, \ldots, l$ are explicitly determined by the expansion of $I$ into a $z^{-1}$-series.
\end{theorem}

A priori $\hat{q}_i(q)$, $i = 1, \ldots, l$, are formal power series in $q$.  In Proposition 5.13 of \cite{iritani07}, Iritani proved that indeed the mirror transform is convergent:

\begin{theorem}[Convergence of toric mirror map \cite{iritani07}] \label{conv}
Let $X$ be a semi-Fano toric manifold and let $\hat{q}$ be the toric mirror transform given in Theorem \ref{mir_thm_Giv}.  For every $i=1, \ldots, l$, $\hat{q}_i(q)$ is convergent in a neighborhood of $q = 0$.
\end{theorem}

The $I$-function can be expressed as oscillatory integrals of the Hori-Vafa superpotential (\cite[p.11]{givental98}), and hence is complex analytic.  Combining Theorems \ref{mir_thm_Giv} and \ref{conv}, one deduces that the $J$-function is also complex analytic in a neighborhood of $q=0$.

The instanton-corrected superpotential can be expressed in terms of this mirror transform:

\begin{defn} \label{W^G}
Let $X$ be a semi-Fano toric manifold, and let $\hat{q}(q)$ be the toric mirror transform.  We define
$$ W_q^{\text{PF}} := W_{\hat{q}(q)}^{\circ} = z_1 + \ldots + z_n + \sum_{i=n+1}^{m} \hat{q}^{\alpha_i}(q) z^{v_i}. $$
The superscript `PF' indicates that the superpotential is defined in terms of the mirror map which is calculated by solving Picard-Fuchs equations.
\end{defn}

While Batyrev's original isomorphism (Theorem \ref{Bat_mir_thm}) does not hold true for general semi-Fano toric manifolds, it can be corrected by using the toric mirror transform:
\begin{theorem}[\cite{givental98, cox-katz, G-I11}]
Let $X$ be a semi-Fano toric manifold equipped with a complexified K\"ahler class $\omega_\cpx$.  We take a nef basis $\{\mathbf{p}_i\}_{i=1}^l \subset H^2(X)$ and let $\hat{q}(q)$ be the toric mirror transform given in Theorem \ref{mir_thm_Giv}.  Then
$$QH^* (X,\omega_\cpx) \cong B_X(\hat{q}(q(\omega_\cpx)))$$
where the isomorphism is given by sending the generators
$$\tilde{\mathbf{p}}_i := \sum_{k=1}^l \left.\frac{\partial \log q_k}{\partial \log \hat{q}_i}\right|_{\hat{q}(q(\omega_\cpx))} \mathbf{p}_k \in QH^* (X,\omega_\cpx) $$
to $u_i \in B_X(\hat{q}(q(\omega_\cpx)))$.  In other words $\mathbf{p}_k \in QH^* (X,\omega_\cpx)$ are sent to
$$\sum_{i=1}^l \left.\frac{\partial \log\hat{q}_i}{\partial \log q_k}\right|_{q(\omega_\cpx)} u_i \in B_X(\hat{q}(q(\omega_\cpx))).$$
\end{theorem}

Now combining the above theorem with Proposition \ref{Bat-Jac}, one has

\begin{theorem}[Second form of toric mirror theorem] \label{Giv_thm2}
Let $X$ be a semi-Fano toric manifold equipped with a complexified K\"ahler class $\omega_\cpx$, and let $W^{\text{PF}}$ be the instanton-corrected superpotential in Definition \ref{W^G}.  Then
$$QH^* (X,\omega_\cpx) \cong \Jac (W^{\text{PF}}_{q(\omega_\cpx)}).$$
Moreover, the isomorphism is given by sending the generators $\mathbf{p}_k \in QH^* (X,\omega_\cpx)$ to $q \dd{q} W^{\text{PF}}_{q}$ evaluated at $q = q(\omega_\cpx)$.
\end{theorem}

\subsection{Lagrangian Floer theory of Fukaya-Oh-Ohta-Ono} \label{FOOO}

Another way to write down the instanton-corrected mirror superpotential is by counting stable holomorphic disks, which is a part of the Lagrangian Floer theory for toric manifolds developed by Fukaya-Oh-Ohta-Ono \cite{FOOO1}.  While FOOO's theory works for general compact toric manifolds, we will restrict to the case when $X$ is a semi-Fano toric manifold for simplicity.

Let $\mathbf{T}$ be a regular toric moment fiber of the semi-Fano toric manifold $X$.  For a disk class $\beta \in \pi_2(X,\mathbf{T})$, we have the moduli space $\mathcal{M}_1 (\beta)$ of stable disks with one boundary marked point representing $\beta$.  $\mathcal{M}_1 (\beta)$ is oriented and compact.  Moreover, since non-constant stable disks bounded by $\mathbf{T}$ have Maslov indices at least two, $\mathcal{M}_1 (\beta)$ has no codimension-one boundary (for a nice and detailed discussion of these, the reader is referred to \cite[Section 3]{auroux07}).  The main problem is transversality: the dimension of $\mathcal{M}_1 (\beta)$ can be higher than its expected (real) dimension, which is $n+\mu(\beta)-2$, where $\mu(\beta)$ denotes the Maslov index of $\beta$.  To tackle with this, FOOO considered the obstruction theory and constructed a virtual fundamental class $[\mathcal{M}_1 (\beta)] \in H_n(\mathbf{T})$, so that the integration
$\int_{[\mathcal{M}_1 (\beta)]} \text{ev}^*[\mathrm{pt}]$
makes sense.

\begin{defn} [One-pointed open GW invariant \cite{FOOO1}] \label{open GW}
Let $X$ be a compact semi-Fano toric manifold, and $\torus$ a regular toric moment fiber of $X$.  The one-pointed open GW invariant associated to a disk class $\beta \in \pi_2(X,\torus)$ is defined as
$$n_\beta := \int_{[\mathcal{M}_1 (\beta)]} \text{ev}^*[\mathrm{pt}],$$
where $[\point] \in H^n(\torus)$ is the point class and $\text{ev}:\moduli_1(\beta)\to\torus$ is the evaluation map.
\end{defn}

Note that $\dim[\mathcal{M}_1 (\beta)]=n$ only when $\mu(\beta) = 2$, hence $n_\beta=0$ whenever $\mu(\beta) \not= 2$.  Analogous to Definition \ref{defn_W_HV}, we have the following definition of instanton-corrected superpotential from Lagrangian Floer theory:

\begin{defn} \label{defn_W_FOOO}
Let $X$ be a semi-Fano toric manifold with a toric K\"ahler form $\omega$, and $P$  the corresponding moment map polytope.  The instanton-corrected superpotential mirror to $X$ from the approach of FOOO's Lagrangian Floer theory is defined to be
\begin{align*}
W^{\text{LF}}: P^{\circ} \times M_\real / M &\to \cpx,\\
W^{\text{LF}}(r,\theta) &= \sum_{\beta \in \pi_2(X,\torus)} n_\beta Z_{\beta}(r,\theta)
\end{align*}
in which the summands are
$$ Z_{\beta}(r,\theta) = \exp\left(-\int_{\beta(r)} \omega + 2\pi\consti \pairing{\partial \beta}{\theta} \right). $$
In the above equation, $\partial \beta \in \pi_1(\torus) \cong N$ so that it has a natural pairing with $\theta$. The superscript `LF' refers to Lagrangian Floer theory.
\end{defn}

Analogous to Proposition \ref{write W_HV}, one may simplify the above expression of $W^{\text{LF}}$ as follows:

\begin{prop} \label{rewrite W^FOOO}
Let $X$ be a semi-Fano toric manifold whose generators of rays in its fan are $\{v_i\}_{i=1}^{m}$.  Without loss of generality suppose $\{v_1, \ldots, v_n\}$ generate a cone in its fan, and it gives the complex coordinates $z_i := \exp \left(2\pi\consti \pairing{v_i}{\cdot}\right): M_\cpx/ M \to \cpx$.  Then $W^{\text{LF}}$ can be written as
$$
W^{\text{LF}} = (1+\delta_1)z_1 + \ldots + (1+\delta_n)z_n + \sum_{k=n+1}^m (1+\delta_k) q^{\alpha_k} z^{v_k}
$$
where
\begin{equation}
\delta_k := \sum_{\alpha\not= 0} n_{\beta_k + \alpha} q^\alpha
\end{equation}
in which the summation is over all non-zero $\alpha \in H_2(X)$ represented by rational curves with Chern number $-K_X \cdot \alpha = 0$.  As before, $q^\alpha = \prod_{i=1}^l q_i^{\pairing{\mathbf{p}_i}{\alpha}}$ for a chosen basis $\{\mathbf{p}_i\}_{i=1}^l$ of $H^2(X)$.  Later we may denote $W^{\text{LF}}$ as $W_q^{\text{LF}}$ to emphasize its dependency on the coordinate $q$ of the K\"ahler moduli.
\end{prop}
\begin{proof}
Equip $X$ with a toric K\"ahler form $\omega$.  According to \cite{FOOO1}, $n_{\beta} \not= 0$ only when $\beta = \beta_i + \alpha$ for $i=1,\ldots,m$ and $\alpha \in H_2(X)$ is represented by a rational curve.  Moreover,
\begin{align*}
Z_{\beta_i + \alpha}(r,\theta) &= \exp\left(-\int_{\beta} \omega + 2\pi\consti \pairing{\partial \beta}{\theta} \right) \\
&= \conste^{-\int_{\alpha} \omega} \exp\left(-\int_{\beta_i(r)} \omega + 2\pi\consti \pairing{v_i}{\theta} \right)\\
&= q^{\alpha} Z_i(r,\theta)
\end{align*}
where $Z_i$ is defined by Equation \eqref{Z_i}.  Also, by the results of Cho-Oh \cite{cho06}, we have $n_{\beta_i} = 1$.  Thus
\begin{align*}
W^{\text{LF}} &= \sum_{i=1}^m \left(\sum_{\alpha} n_{\beta_i + \alpha} q^{\alpha}\right) Z_i(r,\theta)\\
&= \sum_{i=1}^m \left(1 + \sum_{\alpha\not=0} n_{\beta_i + \alpha} q^{\alpha}\right) Z_i(r,\theta).
\end{align*}
From the proof of Proposition \ref{write W_HV}, $Z_i = z_i$ for $i=1,\ldots,n$ and $Z_j = q^{\alpha_j} z^{v_j}$ for $j=n+1,\ldots,m$.  The stated expression for $W^{\text{LF}}$ follows.
\end{proof}

We may take a change of coordinates on $z_i$ and rewrite the above expression of $W^{\text{LF}}$ in the same form as $W^{\text{PF}}$ appeared in the last section: By the change of coordinates $z_i \mapsto z_i/(1+\delta_i)$, the superpotential has the expression
\begin{equation} \label{W^FOOO}
\begin{split}
W_q^{\text{LF}} &= z_1 + \ldots + z_n + \sum_{k=n+1}^m \frac{(1+\delta_k) }{\prod_{j=1}^n (1+\delta_j)^{v_k^j}} q^{\alpha_k} z^{v_k} \\
&= z_1 + \ldots + z_n + \sum_{k=n+1}^m (1 + \tilde{\delta}_k) q^{\alpha_k} z^{v_k}
\end{split}
\end{equation}
for some formal series $\tilde{\delta}_k$ in $q_1, \ldots, q_l$ with constant terms equal to zero.  Notice that the coefficients of $W^{\text{LF}}$ are formal power series in $q_1, \ldots, q_l$.  This means $W_q^{\text{LF}}$ is defined only in a formal neighborhood $\text{Spec } \cpx[[q_1,\ldots,q_l]]$ of $q=0$.  Substituting
$$ q_k = \mathrm{T}^{\int_{\Psi_k} \omega} $$
for each $k = 1, \ldots, l$ where $\mathrm{T}$ is a formal variable, the coefficients live in the universal Novikov ring
$$ \Lambda_0 = \left\{ \sum_{i=0}^\infty a_i \mathrm{T}^{\lambda_i}: a_i \in \cpx, \lambda_i \in \real_{\geq 0}, \lim_{i\to\infty} \lambda_i = \infty \right\}.$$

Using this form of the superpotential, FOOO proved that

\begin{theorem}[Mirror theorem by Fukaya-Oh-Ono-Ohta \cite{FOOO1}, \cite{FOOO10b}] \label{mir_thm_FOOO}
Let $X$ be a semi-Fano toric manifold \footnote{FOOO's theory works for big quantum cohomology of general compact toric manifolds as well.  But since we have only written down the superpotential $W^{\text{LF}}$ without bulk deformation in the semi-Fano case, we shall confine ourselves to this special case of their theorem.} and let $q$ be the coordinate on its complexified K\"ahler moduli corresponding to the choice of basis $\{\mathbf{p}_i\}_{i=1}^l \subset H^2(X)$ (see Section \ref{toric}).  Let $W^{\text{LF}}_q$ be the superpotential from Lagrangian Floer theory given by Equation \eqref{W^FOOO}. Then
$$QH^* (X,q) \cong \Jac (W^{\text{LF}}_{q})$$
as algebras \footnote{Indeed they proved more, namely, they are isomorphic as Frobenius algebras.}.
Moreover, the isomorphism is given by sending the generators $\mathbf{p}_k \in QH^* (X,q)$ to $q_k \dd{q_k} W^{\text{LF}}_{q}$.
\end{theorem}

\section{Universal unfolding of superpotentials} \label{unfold}

In the last section we have discussed the Landau-Ginzburg mirror of a toric manifold $X$, which is a holomorphic function $W: (\cpx^\times)^n \to \cpx$.  We have seen that the deformation theory of $W$, which is recorded by its Jacobian ring $\Jac (W)$, captures the enumerative geometry of $X$, namely its quantum cohomology $QH^*(X)$.  In this section we recall some deformation theory in the Landau-Ginzburg side which we will use to prove our main theorem.  We recommend Section 2.2 of Gross' book \cite{gross_book} which is an excellent review on this old subject.  One of the classical literatures is the text \cite{teissier} written by Teissier, and the papers \cite{douai-sabbah1,douai-sabbah2} by Douai and Sabbah give modern treatments and applications of this subject.

Let $f$ be a holomorphic function on a domain $\mathcal{D} \subset (\cpx^\times)^n$.  We are interested in deformations of $f$, which are defined as follows:

\begin{defn}[Unfolding]
Let $f$ be a holomorphic function on a domain $\mathcal{D} \subset (\cpx^\times)^n$.  An unfolding of $f$ is a pair $(U,W)$, where $U$ is an open neighborhood of $0 \in \cpx^l$ for some $l\in\nat$ and $W$ is a holomorphic function defined on $U \times \mathcal{D}$, such that $W|_{\{0\} \times U} = f$.
\end{defn}

Naturally one seeks for a universal object in the categories of unfoldings, in the sense that every unfolding comes from pull-back of this universal object.  The precise definition is as follows:

\begin{defn}[Universal unfolding]
Let $\mathcal{D} \subset \cpx^n$ be a bounded domain and $f: \mathcal{D} \to \cpx$  a holomorphic function.  An unfolding $(U,\mathcal{W})$ of $f$ is universal if it satisfies the following condition: for every unfolding $(V,W)$ of $f$, there exists an open subset $V' \subset V$ containing $0 \in V$ and the following commutative diagram of holomorphic maps:

$$
\begin{diagram} \label{univ_unfold_diag}
\node{V' \times \mathcal{D}} \arrow{e,t}{\Phi} \arrow{s,l}{\mathrm{pr}} \node{U \times \mathcal{D}} \arrow{s,r}{\mathrm{pr}} \\
\node{V'} \arrow{e,b}{\mathcal{Q}} \node{U}
\end{diagram}
$$
satisfying the conditions that
\begin{enumerate}
\item $W|_{V' \times \mathcal{D}} = \mathcal{W} \circ \Phi$.
\item $\mathcal{Q}: V' \to U$ is a holomorphic map with $\mathcal{Q}(0) = 0$ whose induced map on tangent spaces is unique.
\item $\Phi|_{\{0\} \times V'}:\{0\} \times \mathcal{D} \to \{0\} \times \mathcal{D}$ is the identity map.
\end{enumerate}
Here $\mathrm{pr}: V' \times \mathcal{D} \to V'$ and $\mathrm{pr}: U \times \mathcal{D} \to U$ are natural projections to the first factor.
\end{defn}

Now comes the essential point which we will use in the next section: When $f$ is semi-simple, there is an easy way to write down an universal unfolding of $f$.  First of all, let us recall the definition of semi-simplicity (the Jacobian ring is defined as in Definition \ref{Jac_def}):

\begin{defn}[Semi-simplicity]
Let $\mathcal{D} \subset \cpx^n$ be a domain and $f: \mathcal{D} \to \cpx$  a holomorphic function.  Denote its critical points by $\mathrm{cr}_1, \ldots, \mathrm{cr}_N$.  There is a natural homomorphism of algebras $\Jac (f) \to \cpx^N$ by sending $g \in \Jac (f)$ to $\big(g(\mathrm{cr}_1), \ldots, g(\mathrm{cr}_N)\big)$.  $f$ is said to be semi-simple if this is an isomorphism of algebras.
\end{defn}

In the above definition, $\cpx^N$ is equipped with the standard algebra structure:
$$(a_1, \ldots, a_N) \cdot (b_1, \ldots, b_N) := (a_1 b_1, \ldots, a_N b_N).$$

Going back to our situation, we would like to consider the deformation theory of $W^{\text{PF}}_q$ (or $W^{\text{LF}}_q$).  By Corollary 5.12 of Iritani's paper \cite{iritani07}, the quantum cohomology $QH^*(X,q)$ of a projective toric manifold is semi-simple for a generic K\"ahler class $q$.  By Theorem \ref{Giv_thm2} (Second form of toric mirror theorm), $\Jac(W^{\text{PF}}_q) \cong QH^*(X,q)$, and hence

\begin{theorem}[Semi-simplicity of superpotential \cite{iritani07}] \label{ss}
For generic $q$, $W^{\text{PF}}_q$ is semi-simple.
\end{theorem}

Now comes the key point (which is classical) in this section:

\begin{theorem}  \cite{teissier}
Let $\mathcal{D} \subset \cpx^n$ be a domain and $f: \mathcal{D} \to \cpx$  a holomorphic function.  Suppose that $f$ is semi-simple.  Let $g_1, \ldots, g_N$ be a basis of $\Jac (f) \cong \cpx^N$.  Then the function
$\mathcal{W}: \cpx^N \times (\cpx^*)^n \to \cpx$ defined by
$$ \mathcal{W}(Q_1, \ldots, Q_N, z) := f(z) + \sum_{j=1}^N Q_j g_j(z) $$
is a universal unfolding of $f$.
\end{theorem}

In particular, when the critical values of $f$ are pairwise distinct, $\{g_j := f^{j-1}\}_{j=1}^N$ gives a basis of $\Jac (f)$.  Thus
$$ \mathcal{W}(Q, z) := f(z) + \sum_{j=0}^{N-1} Q_j f^j (z) $$
gives a universal unfolding of $f$.

\section{Equivalence between mirror map and Lagrangian Floer approach} \label{FOOO=G}

In Sections \ref{Giv} and \ref{FOOO}, we introduce the superpotentials mirror to a semi-Fano toric manifold written down by using toric mirror maps and Lagrangian Floer theory respectively. Their Jacobian rings are both isomorphic to the small quantum cohomology ring of $X$.  It is natural to conjecture that these two superpotentials are indeed the same.  In this section we prove such a statement, under the technical assumption that the instanton-corrected superpotential constructed from the Lagrangian-Floer approach has convergent coefficients.  Since $W^{\mathrm{PF}}$ is written in terms of the toric mirror transform, while $W^{\mathrm{LF}}$ is written in terms of one-pointed open GW invariants, an interesting consequence of such an equality is that one-pointed open GW invariants can be computed by the toric mirror transform.

We will see in Subsections \ref{surface} and \ref{can_line} that the above convergence assumption is satisfied when $X$ is of dimension two, or when $X = \proj(K_S\oplus\mathcal{O}_S)$ for some toric Fano manifold $S$.  

\subsection{The main theorem}

Let $X$ be a projective semi-Fano toric $n$-fold whose generators of rays in its fan are denoted by $v_1, \ldots, v_m$.  One has $\rank (H_2(X)) = l := m-n$.

Recall that
$$ W^{\text{LF}} = z_1 + \ldots + z_n + \sum_{k=n+1}^m \frac{(1+\delta_k) q^{\alpha_k}}{\prod_{j=1}^n (1+\delta_j)^{v_k^j}} z^{v_k} $$
and
$$ W^{\text{PF}} = z_1 + \ldots + z_n + \sum_{k=n+1}^{m} \hat{q}^{\alpha_k}(q) z^{v_k}. $$
Based on the coefficients of these expressions, we define the following `mirror maps':

\begin{defn}
Let $X$ be a toric semi-Fano $n$-fold whose generators of rays in its fan are denoted by $v_1, \ldots, v_m$.  Define $P^{\text{LF}} = (P^{\text{LF}}_1, \ldots, P^{\text{LF}}_l)$,
$$ P_i^{\text{LF}} := \frac{(1+\delta_{i+n}) q^{\alpha_{i+n}}}{\prod_{j=1}^n (1+\delta_j)^{v_{i+n}^j}} $$
for $i=1, \ldots, l$.  On the other hand, define $P^{\text{PF}} = (P^{\text{PF}}_1, \ldots, P^{\text{PF}}_l)$,
$$ P_i^{\text{PF}} = \hat{q}^{\alpha_{i+n}}(q)$$
for $i=1, \ldots, l$.
\end{defn}

The readers are referred to Sections \ref{Giv} and \ref{FOOO} for the explanations of notations involved in the above definition.  For each $i=1, \ldots, l$, $P_i^{\text{LF}}$ is a formal power series living in the universal Novikov ring $\Lambda_0$.  On the other hand, since $\hat{q}_i(q)$, $i=1, \ldots, l$, are convergent for $\|q\|$ sufficiently small by Iritani's work \cite{iritani07}, $P^{\text{PF}}$ defines a holomorphic map from a neighborhood of $0 \in \cpx^l$ to $\cpx^l$.

Our main theorem can now be stated as follows:

\begin{theorem} \label{main_thm}
Let $X$ be a compact toric semi-Fano manifold, $W^{\text{LF}}_q$ and $W^{\text{PF}}_q$ the instanton-corrected superpotentials by use of Lagrangian Floer theory and toric mirror maps respectively.  Then
$$ W^{\text{LF}}_q = W^{\text{PF}}_q $$
provided that there exists an open polydisk $U_q \subset \cpx^l$ centered at $q = 0 \in \cpx^l$ in which $P^{\text{LF}}$ defined above converges and defines an analytic map $U_q \to \cpx^l$.
\end{theorem}

The proof of the above theorem occupies the rest of this subsection.  The idea is the following: Using Theorem \ref{Giv_thm2} and \ref{mir_thm_FOOO}, $\Jac (W^{\text{LF}}_q) \cong \Jac (W^{\text{PF}}_q)$.  From this together with semi-simplicity, it follows that $W^{\text{LF}}_q$ and $W^{\text{PF}}_q$ have the same critical values.  Then we put them into a universal unfolding and this gives us constant families of critical points linking
those of $W^{\text{PF}}_q$ and $W^{\text{LF}}_q$.  Since they have the same critical values, they indeed correspond to the same based point in the universal family, and hence they are indeed the same.

We always assume that there exists an open polydisk $U_q \subset \cpx^l$ centered at $q = 0 \in \cpx^l$ in which $P^{\text{LF}}$ converges and defines an analytic map $U_q \to \cpx^l$.  By Theorem \ref{conv}, $P^{\text{PF}}$ is convergent in an open neighborhood of $q=0$.  Thus by shrinking $U_q$ if necessary, we may assume that $P^{\text{PF}}$ also defines an analytic map on $U_q$.

Let $W:\cpx^l \times (\cpx^\times)^n \to \cpx$ be the map
\begin{equation} \label{W_P}
W(P, z) := z_1 + \ldots + z_n + \sum_{k=n+1}^m P_{k-n} z^{v_k}.
\end{equation}
Then $W^{\text{LF}}(q,z) = W(P^{\text{LF}}(q),z)$ and $W^{\text{PF}}(q,z) = W(P^{\text{PF}}(q),z)$.  It suffices to prove that $P^{\text{LF}} = P^{\text{PF}}$ on $U_q$.

By Theorem \ref{ss}, $W^{\text{PF}}_q = W_{P^{\text{PF}}(q)}$ is semi-simple for generic $q$.  Moreover $P^{\text{PF}}$ is holomorphic which maps $U_q$ onto an open set of $\cpx^l$ around $P = 0$.  Thus $W_{P}$ is semi-simple for generic $P$.  Let $N$ be the number of critical points of $f$.  For generic $P \in (\cpx^\times)^l$, $W_P$ has $N$ distinct critical values.  Thus we can fix a base-point $\underline{q} = (\underline{q}_1, \ldots, \underline{q}_l) \in U_q$ with $P^{\text{LF}}(\underline{q}), P^{\text{PF}}(\underline{q}) \in (\cpx^\times)^l$ such that $W^{\text{LF}}_{\underline{q}} = W_{P^{\text{LF}}(\underline{q})}$ and $W^{\text{PF}}_{\underline{q}} = W_{P^{\text{PF}}(\underline{q})}$ are semi-simple, and both $W_{P^{\text{LF}}(\underline{q})}$ and $W_{P^{\text{PF}}(\underline{q})}$ have $N$ distinct critical values.

Denote $\underline{P} = P^{\text{LF}} (\underline{q}) \in (\cpx^\times)^l$ and let $f = W_{\underline{P}}$ which is semi-simple and has $N$ distinct critical values.  Let $\mathcal{D} \subset (\cpx^\times)^n$ be a bounded domain containing all the critical points of $f$.  Then $\mathcal{W}: \cpx^N \times \mathcal{D} \to \cpx$ defined by
\begin{equation} \label{curl_W}
\mathcal{W}(Q,\zeta) := f(\zeta) + \sum_{i=0}^{N-1} Q_i \big(f(\zeta)\big)^i
\end{equation}
for $Q = (Q_1, \ldots, Q_{N-1}) \in \cpx^N$ gives a universal unfolding of $f|_{\mathcal{D}}$ (see Section \ref{unfold}).

\begin{prop} \label{same_crit}
Let $f: \mathcal{D} \to \cpx$ be a semi-simple holomorphic function on a bounded domain $\mathcal{D} \subset \cpx^n$ with $\dim(\Jac(f)) = N$, and $\mathcal{W}: \cpx^N \times \mathcal{D} \to \cpx$  the holomorphic function defined by Equation \eqref{curl_W}.  Then there exists an open neighborhood $\Delta_Q$ of $0 \in \cpx^N$ such that for every $Q \in \Delta_Q$, $\mathcal{W}_Q$ and $f$ have the same set of critical points.
\end{prop}

\begin{proof}
From Equation \eqref{curl_W}, one has
$$ \frac{\partial \mathcal{W}(Q,\zeta)}{\partial \zeta_i} = \left( 1 + \sum_{i=0}^{N-1} i\, Q_i f^{i-1}(\zeta) \right) \frac{\partial{f(\zeta)}}{\partial \zeta_i}.$$
In the bounded domain $\mathcal{D}$, $f^{i-1}$ is bounded for all $i = 0, \ldots, N-1$.  Thus there exists an open neighborhood $\Delta_Q$ of $0 \in \cpx^N$ such that for all $Q \in \Delta_Q$ and $\zeta \in \mathcal{D}$,
$1 + \sum_{i=0}^{N-1} i\, Q_i f^{i-1}(\zeta)$ is never zero.  Thus $\frac{\partial \mathcal{W}(Q,\zeta)}{\partial \zeta_i} = 0$ if and only if $\frac{\partial{f(\zeta)}}{\partial \zeta_i} = 0$, and hence $\mathcal{W}_Q$ and $f$ have the same set of critical points.
\end{proof}

On the other hand, $W_P$ gives an unfolding of $f$.  From now on, for all $P \in \cpx^l$ we restrict the domain of $W_P$ to $\mathcal{D}$ and still denote it by $W_P$ to simplify notations.  As a result, there exists an open polydisk $\Delta_P \subset (\cpx^\times)^l$ centered at $\underline{P}$, holomorphic maps $\mathcal{Q}:\Delta_P \to \cpx^N$ and $\zeta:\Delta_P \times \mathcal{D} \to (\cpx^\times)^n$ with $\mathcal{Q} (\underline{P}) = 0$ and $\zeta (\underline{P}, z) = z$ such that
\begin{equation} \label{unfold_eq}
W(P,z) = \mathcal{W}(\mathcal{Q} (P),\zeta (P, z))
\end{equation}
for all $(P,z) \in \Delta_P \times \mathcal{D}$.

\begin{prop} \label{nondeg_dQ}
Let $f: \mathcal{D} \to \cpx$ be a semi-simple holomorphic function on a bounded domain $\mathcal{D} \subset \cpx^n$ with $\dim(\Jac(f)) = N$, and let $\mathcal{W}: \cpx^N \times \mathcal{D} \to \cpx$ be defined by Equation \eqref{curl_W}.  Let $W: U \times \mathcal{D} \to \cpx$ be a holomorphic function with $W_{\underline{P}} = f$ for some $\underline{P} \in U$.  Suppose $\mathcal{Q}:U \to \cpx^N$ and $\zeta:U \times \mathcal{D} \to (\cpx^\times)^n$ are holomorphic functions with $\mathcal{Q} (\underline{P}) = 0$ and $\zeta (\underline{P}, z) = z$ such that
$$W(P,z) = \mathcal{W}(\mathcal{Q} (P),\zeta (P, z)).$$
If $\left\{\left. \dd{P_k} \right|_{\underline{P}} W_P: k=1, \ldots, l\right\}$ is a linearly independent subset of the vector space $\Jac (f)$, then the Jacobian matrix $\frac{\partial\mathcal{Q}}{\partial P}(\underline{P})$ is non-degenerate.
\end{prop}

\begin{proof}
Differentiating both sides of the above equation with respect to $P_k$ ($k=1,\ldots,l$), one has
$$ \frac{\partial W}{\partial P_k} (P,z) = \sum_{i=1}^{N} \frac{\partial \mathcal{W}}{\partial Q_i}(\mathcal{Q} (P),\zeta (P, z)) \frac{\partial\mathcal{Q}_i}{\partial P_k}(P) + \sum_{j=1}^{n} \frac{\partial \mathcal{W}}{\partial \zeta_j}(\mathcal{Q} (P),\zeta (P, z)) \frac{\partial \zeta_j}{\partial P_k}(P,z)$$
for all $(P,z)$.  Now take $P = \underline{P}$ and $z$ to be a critical point $\mathrm{cr}$ of $f$.  Then $\mathcal{Q} (\underline{P}) = 0$ and $\zeta (\underline{P}, \mathrm{cr}) = \mathrm{cr}$, and so
$$\frac{\partial \mathcal{W}}{\partial \zeta_j}(\mathcal{Q} (\underline{P}),\zeta (\underline{P}, \mathrm{cr})) = \frac{\partial f}{\partial z_j}(\mathrm{cr}) = 0$$
for all $j = 1, \ldots, n$.  Thus
$$\frac{\partial W}{\partial P_k} (\underline{P},\mathrm{cr}) = \sum_{i=1}^{N} \frac{\partial \mathcal{W}}{\partial Q_i}(0, \mathrm{cr}) \frac{\partial\mathcal{Q}_i}{\partial P_k}(\underline{P}) = \sum_{i=1}^{N} \big(f(\mathrm{cr})\big)^i \frac{\partial\mathcal{Q}_i}{\partial P_k}(\underline{P})$$
for each critical point $\mathrm{cr}$ of $f$.  Label the critical points of $f$ as $\mathrm{cr}_1, \ldots, \mathrm{cr}_N$.  Then
$$\frac{\partial W}{\partial P_k} (\underline{P},\mathrm{cr}_j) = \sum_{i=1}^{N} \big(f(\mathrm{cr}_j)\big)^i \frac{\partial\mathcal{Q}_i}{\partial P_k}(\underline{P}).$$
Since the critical values of $f$ are pairwise distinct, the square matrix $\big(f^i(\mathrm{cr}_j)\big)_{i,j = 1}^N$ is non-degenerate.  Also by assumption $\left\{\left. \dd{P_k} \right|_{\underline{P}} W_P: k=1, \ldots, l\right\}$ is linearly independent in $\Jac (f) \cong \bigoplus_{j=1}^N \cpx\langle \mathrm{cr}_j \rangle$, and so the matrix $\left(\left. \dd{P_k} \right|_{\underline{P}} W_P (\mathrm{cr}_j)\right)_{j,k}$ is non-degenerate.  This implies that the Jacobian matrix $\frac{\partial\mathcal{Q}}{\partial P}(\underline{P})$ is non-degenerate.
\end{proof}

\begin{prop} \label{li_dW}
Let $X$ be a toric manifold and let $W:\cpx^l \times (\cpx^\times)^n\to \cpx$ be defined by Equation \eqref{W_P}.  Fix $\underline{P} \in (\cpx^\times)^l$ and let $f = W_{\underline{P}}$.  Then $\left\{\left. \dd{P_k} \right|_{\underline{P}} W_P: k=1, \ldots, l\right\}$ is a linearly independent subset in $\Jac (f)$.
\end{prop}

\begin{proof}
Under the isomorphism $ QH^*(X,\omega_{\underline{q}}) \cong \Jac (f) $ given by the Mirror Theorem \ref{mir_thm_FOOO} of FOOO, $\left.\dd{q_i}\right|_{\underline{q}} \in H^2(X)$ is mapped to $\left.\dd{q_i}\right|_{\underline{q}} W^{\text{LF}}_q \in \Jac(f)$ for each $i=1,\ldots,l$.  Thus
$$ \left.\dd{q_i}\right|_{\underline{q}} W^\text{LF}_q = \sum_{k=1}^l \frac{\partial P^\text{LF}_k(\underline{q})}{\partial q_i} \left(\left.\dd{P_k} \right|_{\underline{P}} W_P\right), i=1,\ldots,l$$
is a linearly independent subset of $\Jac(f)$.  This implies the Jacobian matrix $\left(\frac{\partial P^\text{LF}_k(\underline{q})}{\partial q_i}\right)_{i,k=1}^{n}$ is non-degenerate and $\left\{\left. \dd{P_k} \right|_{\underline{P}} W_P: k=1, \ldots, l\right\}$ is a linearly independent subset in $\Jac (f)$.
\end{proof}

\begin{prop} \label{sat_cond}
By contracting $\Delta_P$ and $\mathcal{D}$ if necessary (still requiring that $\mathcal{D}$ contains all the critical points of $f$), we can achieve the following:

\begin{enumerate}
\item $\mathcal{Q}$ is an embedding. \label{embed_cond}
\item $W_P$ is semi-simple for all $P \in \Delta_P$. \label{ss_cond}
\item $\mathcal{W}_{\mathcal{Q}(P)}$ has the same set of critical points as $f$ for all $P \in \Delta_P$. \label{same_crit_cond}
\item $\frac{\partial \zeta}{\partial z}(P,z)$ is non-degenerate for every $P \in \Delta_P$ and $z \in \mathcal{D}$. \label{non-deg_cond}
\end{enumerate}
\end{prop}

\begin{proof}
Combining Propositions \ref{nondeg_dQ} and \ref{li_dW}, $\frac{\partial\mathcal{Q}}{\partial P}(\underline{P})$ is non-degenerate.  Thus $\mathcal{Q}$ is a local embedding around $P = \underline{P}$, and thus Condition \eqref{embed_cond} can be achieved.

$W_P$ is semi-simple at $\underline{P}$, and semi-simplicity is an open condition.  Thus Condition \eqref{ss_cond} can be achieved.

By Proposition \ref{same_crit}, $\mathcal{W}_{Q}$ has the same set of critical points as $f$ in a neighborhood $\Delta_Q$.  The inverse image of $\Delta_Q$ is an open set containing $\underline{P}$, and so Condition \eqref{same_crit_cond} can be achieved.

Since $\zeta_{\underline{P}} = \mathrm{Id}$, $\frac{\partial \zeta}{\partial P}(\underline{P},z)$ is non-degenerate for all $z \in \mathcal{D}$.  We take a compact subset $\bar{\mathcal{D}} \subset \mathcal{D}$ whose interior $\tilde{\mathcal{D}}$ contains all the critical points of $f$, and restrict $\zeta_P$ on $\bar{\mathcal{D}}$. Then for $P$ sufficiently close to $\underline{P}$ and $z \in \bar{\mathcal{D}}$, $\frac{\partial \zeta}{\partial z_i}(P,z)$ is non-degenerate.
\end{proof}

From now on we always take $\Delta_P$ and $\mathcal{D}$ such that all the conditions in Proposition \ref{sat_cond} are satisfied.  The preimage of $\Delta_P$ under $P^{\text{LF}}$ is an open subset of $U_q$ containing the based point $\underline{q}$.  Thus we may take a polydisk $\Delta_q \subset U_q$ centered at $\underline{q}$ such that $P^{\text{LF}}(\Delta_q) \subset \Delta_P$.

Let
$$\hat{\Delta}_P := \{(P,z) \in \Delta_P \times \mathcal{D}: z \text{ is a critical point of } W_P\}.$$
$\hat{\Delta}_P$ is an $N$-fold cover of $\Delta_P$ by projection to the first coordinate, where $N$ is the number of critical points of $f = W_{\underline{P}}$.  Moreover since $W_P$ is semi-simple for all $P \in \Delta_P$, $\hat{\Delta}_P$ consists of $N$ connected components, each containing a point of the form $(\underline{P}, \mathrm{cr}_0)$ where $\mathrm{cr}_0$ is a critical point of $f$.  Thus for every critical point $\mathrm{cr}(P)$ of $W_P$, there exists a unique critical point $\mathrm{cr}_0$ of $f$ such that $(P,\mathrm{cr}(P))$ and $(\underline{P},\mathrm{cr}_0)$ are lying in the same connected component.  We say that $\mathrm{cr}_0$ is the critical point of $f$ corresponding to $\mathrm{cr}(P)$.

\begin{lemma} \label{crit_to_crit}
For $P \in \Delta_P$, let $\mathrm{cr}(P)$ be a critical point of $W_P$, and let $\mathrm{cr}_0$ be the critical point of $f$ corresponding to $\mathrm{cr}(P)$.  Then
$$ \zeta(P,\mathrm{cr}(P)) = \mathrm{cr}_0. $$
\end{lemma}

\begin{proof}
Differentiating the equation \eqref{unfold_eq} with respect to $z_i$ for $i=1,\ldots,n$, we have
$$ \frac{\partial W}{\partial z_i} (P,z) = \sum_{j=1}^{n} \frac{\partial \mathcal{W}}{\partial \zeta_j} \big(\mathcal{Q}(P),\zeta (P,z)\big) \frac{\partial \zeta_j}{\partial z_i}(P,z).$$
Let $z = \mathrm{cr}(P)$ be a critical point of $W_P$.  Then the left hand side is equal to zero for all $i=1,\ldots,n$.  By Condition \eqref{non-deg_cond} of Proposition \ref{sat_cond}, $\left(\frac{\partial \zeta_j}{\partial z_i}(P,z)\right)_{i,j=1}^n$ is non-degenerate.  Thus for all $j=1,\ldots,n$, we have
$$\frac{\partial \mathcal{W}}{\partial \zeta_j} \big(\mathcal{Q}(P),\zeta (P,\mathrm{cr}(P))\big) = 0,$$
meaning that $\zeta (P,\mathrm{cr}(P))$ is a critical point of $\mathcal{W}_{\mathcal{Q}(P)}$.  By Proposition \ref{same_crit}, $\mathcal{W}_{\mathcal{Q}(P)}$ has the same set of critical points of $f$.  Thus for every $P$, $\zeta (P,\mathrm{cr}(P))$ is a critical point of $f$.

A path $\gamma:[0,1] \to \Delta_P$ joining $P$ and $\underline{P}$ lifts to a path $\hat{\gamma}:[0,1] \to \hat{\Delta}_P$ joining $(P,\mathrm{cr}(P))$  to $(\underline{P}, \mathrm{cr}_0)$, where $\mathrm{cr}_0$ is the critical point of $f$ corresponding to $\mathrm{cr}(P)$.  Now for every $t$, $\hat{\gamma}(t) = (\gamma(t), z)$, where $z$ is a critical point of $W_{\gamma(t)}$.  By the above deduction $\zeta \circ \hat{\gamma}(t)$ is a critical point of $f$.  But the critical points of $f$ are isolated, which forces $\zeta \circ \hat{\gamma}$ to be constant.  Thus
$$\zeta (P,\mathrm{cr}(P))=\zeta \circ \hat{\gamma}(0) = \zeta \circ \hat{\gamma}(1) = \zeta(\underline{P}, \mathrm{cr}_0) = \mathrm{cr}_0.$$
\end{proof}

\begin{prop} \label{W_curlW}
Let $U$ be a contractible open subset of $(\cpx^\times)^l$ containing $\Delta_P \ni \underline{P}$ such that for all $P \in U$, $W_P$ is semi-simple.  Let $\{\mathrm{cr}_i(P)\}$ be the set of critical points of $W_P$, where $\mathrm{cr}_i:U \to (\cpx^\times)^n$ are holomorphic maps.  Denote the critical points of $f$ corresponding to $\mathrm{cr}_i(P)$ by $\underline{\mathrm{cr}}_i$.

Then the holomorphic map $\mathcal{Q}: \Delta_P \to \cpx^N$ extends to $U$.  Moreover
$$ W(P,\mathrm{cr}_i(P)) = \mathcal{W}(\mathcal{Q}(P), \underline{cr}_i) $$
for all $P \in U$ and $i = 1, \ldots, N$.
\end{prop}

\begin{proof}
We have
$$ W(P,z) = \mathcal{W}(\mathcal{Q} (P),\zeta (P, z)) $$
for all $P \in \Delta_P$ and $z \in \mathcal{D}$.  Now take $z = \mathrm{cr}_i(P)$.  By Lemma \ref{crit_to_crit}, $\zeta(P,\mathrm{cr}_i(P)) = \underline{cr}_i$.  Thus for all $P \in \Delta_P$,
\begin{align*}
W(P,\mathrm{cr}_i(P)) &= \mathcal{W}(\mathcal{Q}(P), \underline{cr}_i) \\
&= f(\underline{cr}_i) + \sum_{k=0}^{N-1} f^k(\underline{cr}_i) \mathcal{Q}_k(P).
\end{align*}
Since $f$ has pairwise distinct critical values, the matrix $M = \big(f^k(\underline{cr}_i)\big)$ is invertible.  Thus the above equation determines $\mathcal{Q}_k(P)$, which extends to define $\mathcal{Q}$ on $U$.
\end{proof}

From now on, we take $\mathrm{Dom}_{\mathcal{Q}}$ to be a contractible open subset of $(\cpx^\times)^l$ with $\Delta_P$ and $P^{\text{PF}}(\underline{q})$ such that for all $P \in U$, $W_P$ is semi-simple and has pairwise distinct critical values.  By the above proposition $\mathcal{Q}$ extends to be defined on $U_{\mathcal{Q}}$.

\begin{lemma} \label{identity}
Let $A = (\cpx^N, \cdot)$ be an algebra where the multiplication is given by
$$(a_1, \ldots, a_N) \cdot (b_1, \ldots, b_N) = (a_1 b_1, \ldots, a_N b_N).$$
Denote by $\{e_i\}_{i=1}^N$ the standard basis of $\cpx^N$.  If $\Phi: A \to A$ is an isomorphism of Frobenius algebras, then $\Phi$ is a permutation matrix written in terms of the basis $\{e_i\}_{i=1}^N$, that is, there exists a permutation $\sigma$ on $\{1, \ldots, N\}$ such that $\Phi (e_i) = e_{\sigma(i)}$ for all $i = 1, \ldots, N$.
\end{lemma}

\begin{proof}
Notice that $e_i \cdot e_i = e_i$ and $e_i \cdot e_j = 0$ for $i \not= j$.  Since $\Phi$ preserves the product structure, one has
\begin{align*}
\Phi(e_i) \cdot \Phi(e_i) &= \Phi(e_i); \\
\Phi(e_i) \cdot \Phi(e_j) &= 0 \text{ for } i \not= j.
\end{align*}
Let $\Phi(e_i) = (\Phi_i^1, \ldots, \Phi_i^N)$.  The first equation implies $(\Phi_i^j)^2 = \Phi_i^j$ for all $i, j = 1, \ldots, N$, which forces $\Phi_i^j$ to be either $0$ or $1$.  Then the second equation implies there exists a permutation $\sigma$ on $\{1, \ldots, N\}$ such that $\Phi (e_i) = e_{\sigma(i)}$ for all $i = 1, \ldots, N$.
\end{proof}

\begin{prop} \label{der_eq}
For all $q\in\Delta_q \subset \cpx^{\,l}$, $i=1,\ldots,l$ and $j=1,\ldots,N$, we have
$$ \dd{q_i} \mathcal{W}(\mathcal{Q}^{\text{LF}}(q), \underline{\mathrm{cr}}_j) = \dd{q_i} \mathcal{W}(\mathcal{Q}^{\text{PF}}(q), \underline{\mathrm{cr}}_j).$$
\end{prop}

\begin{proof}
Combining the toric mirror theorem and FOOO's mirror theorem, one has the isomorphism of algebras
$$ \Jac (W_q^{\text{LF}}) \cong QH^* (X, q) \cong \Jac (W_q^{\text{PF}}) $$
where the isomorphism $QH^* (X, \omega_q) \cong \Jac (W_q^{\text{LF}})$ is given by sending $v \in QH^* (X, \omega_q)$ to $\partial_v W_q^{\text{LF}}$, and the isomorphism $QH^* (X, \omega_q) \cong \Jac (W_q^{\text{PF}})$ is given by sending $v \in QH^* (X, \omega_q)$ to $\partial_v W_q^{\text{PF}}$.  (Here $q$ is a point in $H^*(X)$ and $v$ is a tangent vector at $q \in H^*(X)$, and so the directional derivatives make sense.)

Moreover, since $W_q^{\text{LF}} = W_{P^{\text{LF}}(q)}$ is semi-simple, one has $\Jac (W_q^{\text{LF}}) \cong \cpx^N$ as algebras,  where the isomorphisms are given by evaluations at critical points, that is, sending $f \in \Jac (W_q^{\text{LF}})$ to $\big(f(\mathrm{cr_1(P^{\text{LF}}(q))}), \ldots, f(\mathrm{cr_N(P^{\text{LF}}(q))})\big) \in \cpx^N$. Similarly $\Jac (W_q^{\text{PF}}) \cong \cpx^N$ by sending $f \in \Jac (W_q^{\text{PF}})$  to $\big(f(\mathrm{cr_j(P^{\text{PF}}(q))})\big)_{j=1}^N \in \cpx^N$.  Together with the above isomorphism $\Jac (W_q^{\text{LF}}) \cong \Jac (W_q^{\text{PF}})$, this gives an isomorphism $\cpx^N \to \cpx^N$ as algebras, which must be a permutation matrix $\sigma$ by Lemma \ref{identity}.  In particular, since $\Psi_i \in QH^* (X, \omega_q)$ is mapped to $\dd{q_i} W_q^{\text{LF}} \in \Jac (W_q^{\text{LF}})$ and to $\dd{q_i} W_q^{\text{PF}} \in \Jac (W_q^{\text{PF}})$, we have
$$\dd{q_i} W\big(P^{\text{LF}}(q), \mathrm{cr}_j(P^{\text{LF}}(q))\big) = \dd{q_i} W\big(P^{\text{PF}}(q), \mathrm{cr}_{\sigma(j)}(P^{\text{PF}}(q))\big).$$
Since the leading order terms of $P^{\text{LF}}$ and $P^{\text{PF}}$ are equal to each other, $\sigma$ must be the identity.  Using Proposition \ref{W_curlW}, we have
$$ \dd{q_i} \mathcal{W}(\mathcal{Q}^{\text{LF}}(q), \underline{\mathrm{cr}}_j) = \dd{q_i} \mathcal{W}(\mathcal{Q}^{\text{PF}}(q), \underline{\mathrm{cr}}_j).$$
\end{proof}

Proposition \ref{der_eq} gives
$$ \sum_{i=0}^{N-1} \frac{\partial\mathcal{Q}^{\text{LF}}_i}{\partial q} f^i (\underline{\mathrm{cr}}_j) = \sum_{i=0}^{N-1} \frac{\partial\mathcal{Q}^{\text{PF}}_i}{\partial q} f^i (\underline{\mathrm{cr}}_j). $$
Let us denote the critical values of $f$ by $\mathrm{cv}_j$, $j = 1, \ldots, N$, which are pairwise distinct.  Since the matrix $(\mathrm{cv}^i_j)_{i,j = 1, \ldots, N}$ is non-degenerate, the above equality implies that
$$\frac{\partial\mathcal{Q}^{\text{LF}}}{\partial q}
=\frac{\partial\mathcal{Q}^{\text{PF}}}{\partial q}
$$
and so $\mathcal{Q}^{\text{LF}} - \mathcal{Q}^{\text{PF}}$ is a constant.  As $q \to 0$, both $P^{\text{LF}}(q)$ and $P^{\text{PF}}(q)$ tend to $0$.  Thus $\mathcal{Q}^{\text{LF}}(q) - \mathcal{Q}^{\text{PF}}(q) = \mathcal{Q}(P^{\text{LF}}(q)) - \mathcal{Q}(P^{\text{PF}}(q))$ can only be $0$.  In particular when we put $q = \underline{q}$, $\mathcal{Q}(P^{\text{LF}}(\underline{q})) = \mathcal{Q}(P^{\text{PF}}(\underline{q})) = 0$, and so $W_{q_0}^{\text{PF}} = W_{q_0}^{\text{LF}}$.  Since $q_0$ is arbitrary, this finishes the proof of Theorem \ref{main_thm}.

\subsection{Application to computing open GW invariants}\label{sec:computing_openGW}

Now we deduce enumerative consequences of the equality $W^{\text{LF}} = W^{\text{PF}}$ (see Theorem \ref{main_thm}). In particular, we obtain a very powerful method to effectively compute {\em all} the open GW invariants of semi-Fano toric manifolds.

To begin with, recall that the equality gives
$$ \frac{1+\delta_k }{\prod_{j=1}^n (1+\delta_j)^{v_k^j}} q^{\alpha_k} = \hat{q}^{\alpha_k} (q) $$
for $k=n+1,\ldots,m$, where the left hand side is a coefficient of $W^{\text{LF}}$ and the right hand side is the corresponding coefficient of $W^{\text{PF}}$.  The toric mirror transform $\hat{q}(q)$ can be computed explicitly as follows \cite{givental98}.  Expand the $I$-function as
$$ I(\hat{q},z)=z \, \conste^{(\mathbf{p}_1 \log\hat{q}_1 + \ldots + \mathbf{p}_l \log\hat{q}_l) /z} \left( 1 + \frac{1}{z} \sum_{j=1}^l \mathbf{p}_j f_j(\hat{q}) + \mathrm{o}(z^{-1}) \right).$$
Then the inverse mirror map is given by
$$ q_i = \conste^{f_i(\hat{q})} \hat{q}_i $$
and the mirror map is obtained by inverting the above formal power series.  It is of the form $\hat{q}_i = \conste^{\phi_i(q)} q_i$ for $i=1,\ldots,l$.  Thus the above equality gives
\begin{equation} \label{to solve}
\prod_{j=1}^m (1+\delta_j)^{\pairing{D_j}{\alpha_k}} = \exp \left(\sum_{i=1}^l \pairing{\mathbf{p_i}}{\alpha_k} \phi_i(q) \right)
\end{equation}
for $k=n+1,\ldots,m$.  Now recall that
$$ \delta_j = \sum_{\alpha\not= 0} n_{\beta_j + \alpha} q^\alpha $$
is expressed in terms of open GW invariants $n_{\beta_j + \alpha}$.  Thus one would like to solve for $\delta_j$ as a series in $q$ in order to compute the open GW invariants.  Notice that we have $l$ equations ($k$ runs from $n+1$ to $m$), while we have $m$ unknown variables $\delta_k$'s! It turns out that at most $l-1$ of the $\delta_k$'s are non-zero.  To see this we need the following result in Gonz\'alez-Iritani \cite{G-I11}:

\begin{prop}[Proposition 4.3 of \cite{G-I11}] \label{gi=0}
Let $X$ be a compact semi-Fano toric manifold and denote its irreducible toric divisors by $D_1, \ldots, D_m$.  Define
\begin{equation}\label{gi}
g_0^{(i)}(\hat{q}_1, \ldots, \hat{q}_l) := \sum_{\substack{\pairing{-K_X}{d} = 0 \\ \pairing{D_i}{d} < 0 \\ \pairing{D_j}{d} \geq 0 \,\, \forall j\neq i }} \frac{(-1)^{\pairing{D_i}{d}}(-\pairing{D_i}{d}-1)!}{\prod_{j\not=i}\pairing{D_j}{d}!} \hat{q}^d.
\end{equation}
Then $g_0^{(i)} = 0$ if and only if the primitive generator $v_i$ corresponding to $D_i$ is a vertex of the fan polytope (the convex hull of primitive generators of rays in the fan) of $X$.
\end{prop}

From the above proposition, we deduce that

\begin{corollary}\label{delta_i=0}
Assume the notations as in Proposition \ref{rewrite W^FOOO}.  At most $l-1$ of the $\delta_i$'s are non-zero, where $l = m-n$ is the dimension of $H^2(X)$.  Moreover, those divisors $D_i$ whose corresponding $\delta_i$ are non-zero are linearly independent in $H^2(X,\rat)$.
\end{corollary}

\begin{proof}
By Proposition \ref{gi=0}, each vertex of the fan polytope corresponds to an index $i=1, \ldots, m$ such that $g_0^{(i)} = 0$.  A non-degenerate $n$-dimensional polytope has at least $(n+1)$ vertices, and hence at least $(n+1)$ of the $g_i$'s are zero.  Thus at most $m-(n+1) = l-1$ of them are non-zero.

Suppose $\delta_i \neq 0$.  Then there exists $\alpha \in H_2(X)$ represented by a rational curve with Chern number zero such that $n_{\beta_i + \alpha} \not= 0$.  The class $\alpha$ is represented by a tree $C$ of genus 0 curves in $X$. Let $C'$ be the irreducible component of $C$ which intersects with the disk representing $\beta_i$. Let $d=[C']\in H_2(X)$. Then the Chern number of $d$ is also zero since $X$ is semi-Fano. Furthermore, $D_i\cdot d<0$ because the invariance of $n_{\beta_i+\alpha}$ under deformation of the Lagrangian torus fiber $L$ implies that $C'$ is contained inside the toric divisor $D_i$. We claim that $D_j\cdot d\geq0$ for all $j\neq i$. When $n=2$, this is obvious. When $n\geq3$, $D_j\cdot d<0$ for some other $j\neq i$ implies that the curve $C'$ is contained inside the codimension two subvariety $D_i\cap D_j$. However, the intersection of $C'$ with the disk representing $\beta_i$ cannot be inside $D_i\cap D_j$ since $\beta_i$ is of Maslov index two. So we conclude that $D_j\cdot d\geq0$ for all $j\neq i$.
Thus $d = [C'] \in H_2(X)$ satisfies the properties that
\begin{align*}
\pairing{-K_X}{d} &= 0 \\
\pairing{D_i}{d} &< 0 \\
\pairing{D_j}{d} &\geq 0 \text{ for all } j\not=i.
\end{align*}

This contributes to a term of $g_0^{(i)}$, and hence $g_0^{(i)} \neq 0$ (distinct $d$ leads to distinct $\hat{q}^d$, and hence they do not cancel each other).  But there are at most $(l-1)$ non-zero $g_0^{(i)}$'s.  It follows that at most $(l-1)$ of the $\delta_i$'s are non-zero.

Let $I \subset \{1, \ldots, m\}$ be the collection of indices such that $\delta_i \not= 0$.  By the above argument, $\{v_i: i \in \{1,\ldots,m\} - I\}$ is the set of vertices of the fan polytope.  Suppose that $\{D_i: i \in I\}$ is linearly dependent in $H^2(X,\rat)$.  Then there exists $\nu \in M-\{0\}$ such that $\pairing{\nu}{v_i} = 0$ for all $i \not\in I$.  However this is impossible since $\{v_i: i \in \{1,\ldots,m\} - I\}$ spans the whole $N_\real$.  Thus $\{D_i: i \in I\}$ is linearly independent in $H^2(X,\rat)$.
\end{proof}

Let $I = \{i_1 < \ldots < i_K \} \subset \{1, \ldots, m\}$ be the collection of indices $i$ such that $v_i$ is not a vertex of the fan polytope.  The above corollary says that $\delta_i \not= 0$ only when $i \in I$.  Complete $\{D_i: i \in I\}$ into a basis of $H^2(X,\rat)$.  Denote its dual basis by $\{\tilde{\Psi}_1, \ldots, \tilde{\Psi}_l \} \subset H_2(X,\rat)$.  From Equation \eqref{to solve},
$$ 1 + \delta_{i_k} = \prod_{j=1}^m (1+\delta_j)^{\pairing{D_j}{\tilde{\Psi}_k}} = \exp \left(\sum_{i=1}^l \pairing{\mathbf{p_i}}{\tilde{\Psi}_k} \phi_i(q) \right) $$
for $k=1,\ldots, K$.  From this we obtain all the one-pointed open GW invariants of a Lagrangian toric fiber of $X$.

\subsection{Examples}
In this section we discuss our main Theorem \ref{main_thm} and its consequence on computation of open GW invariants in several examples.
\subsubsection{Semi-Fano toric surfaces} \label{surface}
In the paper \cite{chan-lau} by the first and the second authors, the open GW invariants and superpotenials $W^{\text{LF}}$ of semi-Fano toric surfaces have been computed.  The main result is the following:

\begin{theorem}[The instanton-corrected superpotential in surface case \cite{chan-lau}] \label{surface_thm}
Let $X$ be a compact semi-Fano toric surface. Let $\beta\in\pi_2(X,\mathbf{T})$ be a class of disks with Maslov index two bounded by a Lagrangian torus fiber $\mathbf{T}$. Then the genus 0 one-point open GW invariant $n_\beta$ is either one or zero according to whether $\beta$ is admissible or not.

As a consequence,
$$
W^{\text{LF}} = (1+\delta_1)z_1 + \ldots + (1+\delta_n)z_n + \sum_{k=n+1}^m (1+\delta_k) q^{\alpha_k} z^{v_k}
$$
where
$$
\delta_k = \sum_{\beta_k + \alpha \text{ is admissible}} q^\alpha
$$
\end{theorem}

The admissibility condition appeared in the above theorem is combinatoric in nature and is defined as follows:

\begin{defn}[Admissibility of disks in surface case.]
Let $X$ be a compact semi-Fano toric surface and denote a regular toric fiber by $\torus$.  A class $\beta\in\pi_2(X,\torus)$ is admissible iff $\beta=b+\sum_k s_kD_k$, where
\begin{enumerate}
\item $b\in\pi_2(X,\torus)$ is a basic disk class intersecting $D_0$ once;
\item $D_k$'s are toric divisors which form a chain of $(-2)$-curves.  In particular the summation in the above equation is finite;
\item Both $s_0\ge s_1\ge s_2\ge\cdots$ and $s_0\ge s_{-1}\ge s_{-2}\ge\cdots$ are nondecreasing integer sequences with $|s_k-s_{k+1}|=0$ or $1$ for each $k$, and the last term of each sequence is not greater than one.
\end{enumerate}
\end{defn}

From the above definition, it follows that the number of admissible disks is finite.  Thus $\delta_k = \sum_{\beta_k + \alpha \text{ is admissible}} q^\alpha$ is just a finite sum, and it follows that $W^{\text{LF}}$ converges.  By Theorem \ref{main_thm}, we have the

\begin{corollary}
Let $X$ be a compact semi-Fano toric surface.  Then
$$ W^{\text{PF}} = W^{\text{LF}}.$$
\end{corollary}
Thus \eqref{PF=LF} holds unconditionally in this case.

We remark that since the toric mirror transform is written down in terms of series in the K\"ahler parameters $q$, checking the above equality by brute force requires non-trivial techniques on handling infinite seires.

\subsubsection{Canonical line bundles of toric Fano manifolds} \label{can_line}
Another class of semi-Fano toric manifolds that we can check convergence of $W^{\text{LF}}$ is $X = \proj(K_S \oplus \mathcal{O}_S)$, where $S$ is a toric Fano manifold.

\begin{theorem} \label{K_S case}
Let $X = \proj(K_S \oplus \mathcal{O}_S)$, where $S$ is a compact toric Fano manifold.  Then $W^{\text{LF}}(q,z)$ converges and defines a holomorphic function in $\Delta_q \times (\cpx^\times)^n$, where $\Delta_q$ is an open neighborhood of $0 \in \cpx^l$.  Combining with Theorem \ref{main_thm}, $$W^{\text{LF}} = W^{\text{PF}}.$$
\end{theorem}

To prove this, we will use the following result proved by the first author \cite{Chan10} on equating open GW invariants with some closed GW invariants, and computation of toric mirror transform.

\begin{theorem}[Open and closed GW invariants \cite{Chan10}] \label{open_closed}

Let $\torus$ be a regular toric fiber of $X = \proj(K_S \oplus \mathcal{O}_S)$, where $S$ is a toric Fano manifold, and let $\beta_0 \in \pi_2(X,\torus)$ be the basic disk class which intersects the zero section once.  Then for every $\alpha \in H_2(X)$ represented by a rational curve of Chern number zero,
$$ n_{\beta_0+\alpha} = \langle \point \rangle_{0,1,h+\alpha} $$
where $h \in H_2(X)$ is the fiber class and $\point \in H^n(X)$ is the Poincar\'e dual of a point in $X$.
\end{theorem}

\begin{proof}[Proof of Theorem \ref{K_S case}]
Denote the generators of rays in the fan of $S$ by $u_1, \ldots, u_m \in N$, and without loss of generality assume that $u_1, \ldots, u_n$ generates a cone in the fan.  Then the generators of rays in the fan of $X$ is $v_0 = (0,1), v_\infty = (0,-1), v_i = (u_i, 1) \in N \times \integer$.  $W^{\text{LF}}$ in this case is given by
$$W^{\text{LF}} = (1+\delta_0) z_0 + q_0 z_0^{-1} + z_1 + \ldots + z_n + \sum_{k=n+1}^m  q^{\alpha_k} z^{v_k}$$
where $q_0 = \conste^{-A}$, $A$ is the symplectic area of the fiber class; $q_1, \ldots, q_l$ are the K\"ahler parameters of $S$ corresponding to a choice of basis $\{\mathbf{p}_1, \ldots, \mathbf{p}_l\}$ of $H^2(S)$, and $\alpha_k$ is defined by Equation \eqref{alpha}.  It suffices to prove that
$$ \delta_0 = \sum_{\alpha \in H_2^\mathrm{eff}(X)\setminus\{0\}} q^{\alpha} n_{\beta_0 + \alpha} $$
converges.  By Theorem \ref{open_closed},
$$ \delta_0 = \sum_{\alpha \in H_2^\mathrm{eff}(X)\setminus\{0\}} q^{\alpha} \langle \point \rangle_{0,1,h+\alpha}. $$

The right hand side in the above theorem are closed GW invariants, which can be found in the coefficients of $J$-function as follows.  From Equation \eqref{J-function}, one has
\begin{align*}
J(q,z) &= z \, \conste^{(\mathbf{p}_1 \log q_1 + \ldots + \mathbf{p}_l \log q_l) /z} \left(1+\sum_{\substack{\alpha, \\ d \in H_2^\mathrm{eff}(X)\setminus\{0\}}} q^d \Big\langle1,\frac{\phi_\alpha}{z-\psi}\Big\rangle_{0,2,d}\phi^\alpha\right) \\
&= z \, \conste^{(\mathbf{p}_1 \log q_1 + \ldots + \mathbf{p}_l \log q_l) /z} \left(1+\sum_{\substack{\alpha, \\ d \in H_2^\mathrm{eff}(X)\setminus\{0\}}} \frac{q^d}{z} \sum_{k\geq 0} \left(\langle 1,\phi_\alpha \psi^k \rangle_{0,2,d} \frac{\phi^\alpha}{z^k}\right) \right) \\
&= z \, \conste^{(\mathbf{p}_1 \log q_1 + \ldots + \mathbf{p}_l \log q_l) /z} \left(1+\sum_{\substack{\alpha, \\ d \in H_2^\mathrm{eff}(X)\setminus\{0\}}} \frac{q^d}{z} \sum_{k\geq 1} \left(\langle \phi_\alpha \psi^{k-1} \rangle_{0,1,d} \frac{\phi^\alpha}{z^k}\right) \right)
\end{align*}
where the last equality follows from the string equation.  Each coefficient of $J$ takes value in $H^*(X)$, which can be written as linear combinations of $\phi^\alpha$'s.  By taking the term corresponding to $k=1$ in the above summation and consider the component $\phi_\alpha = \point$ (so that $\phi^\alpha = 1$, the fundamental class), It follows that the component of fundamental class in the $1/z$-term of $J$ is
$$ \sum_{d \in H_2^\mathrm{eff}(X)\setminus\{0\}} q^d \langle \point \rangle_{0,1,d}. $$
By dimension counting, $\langle \point \rangle_{0,1,d} \not= 0$ only when $\dim \moduli_{0,1} (d) = c_1 (d) + n + 1 - 3 = n$, that is, $c_1(d) = 2$, and this holds if and only if $d=h+\alpha$ for some $\alpha$ represented by rational curves with $c_1(\alpha) = 0$.  Thus the above expression is equal to
$$ q^h \sum_{\alpha \in H_2^\mathrm{eff}(X)\setminus\{0\}} q^\alpha \langle \point \rangle_{0,1,h+\alpha} = q^h \delta_0.$$

Since the coefficients of $J$ are convergent (see Theorem \ref{conv} and the remark below there), in particular $\delta_0$ is convergent. Hence the result follows from our main theorem.
\end{proof}

\subsubsection*{The Hirzebruch surface $\mathbf{F}_2$.}
As the simplest non-trivial example\footnote{This example was discussed in a meeting at MIT in June 2009 participated by M. Abouzaid, T. Coates, H. Iritani, and H.-H. T.}, let us consider the Hirzebruch surface $X=\mathbf{F}_2$.  This was the first example that we verified $W^{\mathrm{PF}} = W^{\mathrm{LF}}$ by direct computations.

Its fan consists of four rays which are generated by $v_1 = (1,0), v_2 = (0,1), v_3 = (-1,-2), v_4 = (0,-1)$ respectively.  Let $D_i$ denote the corresponding irreducible toric divisors and $\beta_i$ denote the corresponding basic disk classes.  Take $\Psi_1 = D_4$, $\Psi_2 = D_1$ to be the basis of $H_2(X)$, and denote the corresponding K\"ahler parameters by $q_1$ and $q_2$.

The superpotential via PDE approach is
$$ W^{\mathrm{PF}} = z_1 + z_2 + \check{q}_1(q) \check{q}_2^2(q) z_1^{-1} z_2^{-2} + \check{q}_2(q) z_2^{-1}$$
where $(\check{q}_1(q), \check{q}_2(q))$ is the inverse of toric mirror map.  Let
$$h(x) = \sum_{k > 0} x^k \frac{(-1)^{2k-1}(2k-1)!}{(k!)^2}.$$
Then the toric mirror map is given by
\begin{align*}
q_1 &= \check{q}_1 \exp (-2 h(\check{q}_1)); \\
q_2 &= \check{q}_2 \exp h(\check{q}_1).
\end{align*}
By inverting the first equality, one has
$$\exp\left(-h(\check{q}_1(q))\right) = 1+q_1$$
and so
\begin{align*}
W^{\mathrm{PF}} &= z_1 + z_2 + q_1 q_2^2 z_1^{-1} z_2^{-2} + q_2 \exp (-h(\check{q}_1(q))) z_2^{-1} \\
&= z_1 + z_2 + q_1 q_2^2 z_1^{-1} z_2^{-2} + q_2  z_2^{-1} (1+q_1).
\end{align*}
On the other hand,
$$W^{\mathrm{LF}} = z_1 + z_2 + q_1 q_2^2 z_1^{-1} z_2^{-2} + q_2 z_2^{-1} (1+\delta(q))$$
where
$$\delta(q) = \sum_{k>0} n_{\beta_4 + k \Psi_1} q_1^k.$$
It was shown by Auroux \cite{auroux09}, Fukaya-Oh-Ohta-Ono \cite{FOOO10} and Chan-Lau \cite{chan-lau} that
$n_{\beta_4 + k \Psi_1} = 1$ when $k=0,1$ and zero otherwise.  Thus
$$ W^{\mathrm{LF}} = z_1 + z_2 + q_1 q_2^2 z_1^{-1} z_2^{-2} + q_2  z_2^{-1} (1+q_1).$$
We see that $W^{\mathrm{PF}} = W^{\mathrm{LF}}$.

\subsubsection{A further example}\label{sec:further_example}

To show the enumerative power of \eqref{PF=LF}, we compute an example whose open GW invariants are more complicated than the canonical line bundle of a toric Fano manifold, in the sense that the bubbles which contribute to the open GW invariants are supported by more than one irreducible toric divisors.

The toric data is as follows.  The primitive generators of rays in the fan are
\begin{align*}
&v_1 = (0,0,1), v_2 = (1,0,0), v_3 = (2,0,-1), v_4 = (1,0,-1), \\
&v_5 = (0,1,0), v_6 = (-1,-1,3), v_7 = (0,0,-1).
\end{align*}
and its moment map image (with respect to a toric K\"ahler form) is shown in Figure \ref{blowup_KP2}.  Let $X$ denote this toric manifold.

\begin{figure}[htp]
\begin{center}
\includegraphics[scale=0.8]{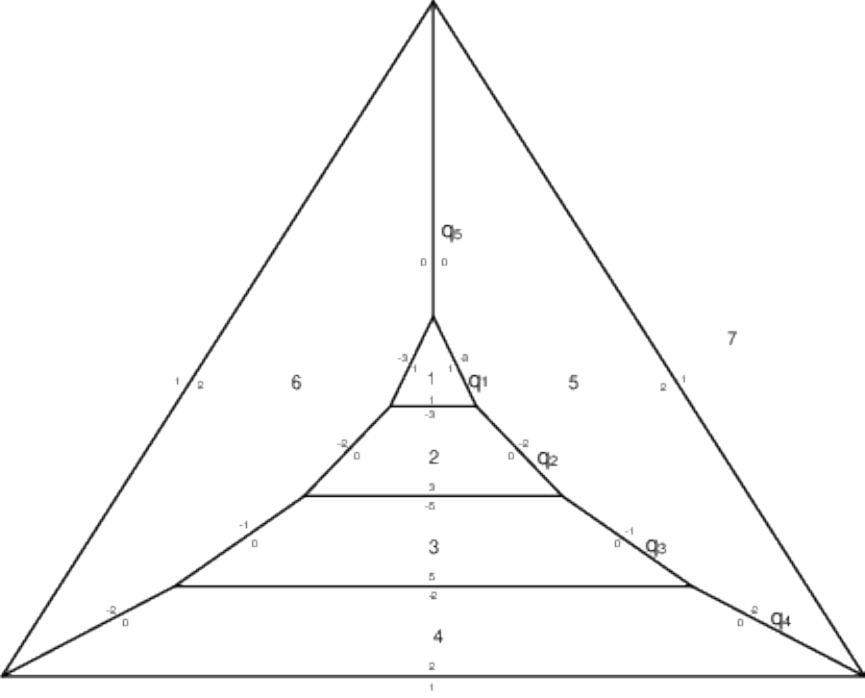}
\end{center}
\caption{A toric semi-Fano manifold which can be obtained by blowing up $\proj(K_{\proj^2}\oplus\mathscr{O}_{\proj^2})$ along lines twice.}
\label{blowup_KP2}
\end{figure}

Each facet is labelled by a number $i = 1, \ldots, 7$ such that its primitive normal vector is given by $v_i$, and the corresponding irreducible toric divisor is denoted by $D_i$.  The numbers beside the edges record the self-intersection numbers of lines inside the corresponding toric surfaces.  We take a basis $\{\Psi_1, \ldots, \Psi_4\} \subset H_2 (X)$, where $\Psi_i = D_i \cap D_5$ for $i = 1, \ldots, 4$.  The corresponding K\"ahler parameters are denoted by $q_j$ for $j = 1, \ldots, 4$ as shown in Figure \ref{blowup_KP2}, and the complex parameters in the mirror are denoted by $\check{q}_j$ for $j = 1, \ldots, 4$.  Let $q_5 = q_2 q_3^2 q_4$ and $\check{q}_5 = \check{q}_2 \check{q}_3^2 \check{q}_4$.

The superpotential from the PDE approach is
$$ W^{\mathrm{PF}} = \check{q}_5 z_3 + \check{q}_1 \check{q}_5^3 z_1 + \check{q}_1^2 \check{q}_2 \check{q}_5^5 z_1^2 z_3^{-1} + \check{q}_3 \check{q}_2 \check{q}_1 \check{q}_5^2 z_1 z_3^{-1} + z_2 + z_1^{-1} z_2^{-1} z_3^3 + z_3^{-1} $$
where $\check{q}(q)$ is the inverse of the toric mirror map.  Let
\begin{align*}
f(x, y) &= \sum_{\substack{k_1 \geq 2 k_2 \geq 0 \\ (k_1,k_2) \not= (0,0) }} x^{k_1}y^{k_2} \frac{(-1)^{3k_1 - k_2 - 1} (3k_1 - k_2 - 1)!}{(k_1!)^2 k_2! (k_1 - 2k_2)!};\\
g(x, y) &= \sum_{\substack{k_2 \geq 3k_1 \geq 0 \\ (k_1,k_2) \not= (0,0) }} x^{k_1}y^{k_2} \frac{(-1)^{2k_2 - k_1 - 1} (2k_2 - k_1 - 1)!}{(k_1!)^2 k_2! (k_2 - 3k_1)!};\\
h(x) &= \sum_{k > 0} x^k \frac{(-1)^{2k-1}(2k-1)!}{(k!)^2}.
\end{align*}
Then the toric mirror map is given by
\begin{align*}
q_1 &= \check{q}_1 \exp (-3f(\check{q}_1,\check{q}_2) + g(\check{q}_1,\check{q}_2)); \\
q_2 &= \check{q}_2 \exp (f(\check{q}_1,\check{q}_2) - 2g(\check{q}_1,\check{q}_2)); \\
q_3 &= \check{q}_3 \exp (g(\check{q}_1,\check{q}_2) + h(\check{q}_4)); \\
q_4 &= \check{q}_4 \exp (-2 h(\check{q}_4)).
\end{align*}

In terms of the inverse mirror map $(\check{q}_1(q), \check{q}_2(q),\check{q}_3(q),\check{q}_4(q))$, where $q = (q_1,q_2,q_3,q_4)$ is a multivariable, one has
\begin{align*}
\check{q}_5 &= q_5 \exp(-f(\check{q}_1(q),\check{q}_2(q))); \\
\check{q}_1 \check{q}_5^3 &= q_1q_5^3 \exp(-g(\check{q}_1(q),\check{q}_2(q))); \\
\check{q}_1^2 \check{q}_2 \check{q}_5^5 &= q_1^2 q_2 q_5^5; \\
\check{q}_3 \check{q}_2 \check{q}_1 \check{q}_5^2 &= q_3 q_2 q_1 q_5^2 \exp(-h(\check{q}_4(q))).
\end{align*}
Thus
\begin{align*}
W^{\mathrm{PF}} =& q_5z_3 \exp(-f(\check{q}_1(q),\check{q}_2(q))) + q_1q_5^3 z_1 \exp(-g(\check{q}_1(q),\check{q}_2(q))) + q_1^2 q_2 q_5^5 z_1^2 z_3^{-1} \\
&+ q_3 q_2 q_1 q_5^2 z_1 z_3^{-1} \exp(-h(\check{q}_4(q))) + z_2 + z_1^{-1} z_2^{-1} z_3^3 + z_3^{-1}
\end{align*}

On the other hand, the superpotential from the Lagrangian-Floer approach is
\begin{align*}
W^{\mathrm{LF}} =& q_5 z_3 (1 + \delta_1(q_1,q_2)) + q_1 q_5^3 z_1 (1 + \delta_2(q_1,q_2)) + q_1^2 q_2 q_5^5 z_1^2 z_3^{-1} \\
& + q_3 q_2 q_1 q_5^2 z_1 z_3^{-1} (1 + \delta_4(q_4)) + z_2 + z_1^{-1} z_2^{-1} z_3^3 + z_3^{-1}
\end{align*}
where
\begin{align*}
\delta_1(q_1,q_2) &= \sum_{\substack{k_1,k_2\geq 0\\ (k_1,k_2) \not= 0}} n_{\beta_1 + k_1 \Psi_1 + k_2 \Psi_2} q_1^{k_1} q_2^{k_2} \\
\delta_2(q_1,q_2) &= \sum_{\substack{k_1,k_2\geq 0\\ (k_1,k_2) \not= 0}} n_{\beta_2 + k_1 \Psi_1 + k_2 \Psi_2} q_1^{k_1} q_2^{k_2} \\
\delta_4(q_4) &= \sum_{k>0} n_{\beta_4 + k \Psi_4} q_4^{k}
\end{align*}
where we recall that $\beta_i$ are the basic disk classes associated to the irreducible toric divisors $D_i$ for $i = 1, \ldots, 7$.

The equality $W^{\mathrm{LF}} = W^{\mathrm{PF}}$
implies that
\begin{align*}
\delta_1(q_1,q_2) &= \exp(-f(\check{q}_1(q),\check{q}_2(q)));\\
\delta_2(q_1,q_2) &= \exp(-g(\check{q}_1(q),\check{q}_2(q)));\\
\delta_4(q_1,q_2) &=
\exp(-h(\check{q}_4(q))).
\end{align*}

In the above expressions, $\exp(-h(\check{q}_4(q)))$ is the easiest one to write down (it also appears in the case of the Hizerbruch surface $\mathbf{F}_2$):
$$ \delta_4(q_1,q_2) = \exp(-h(\check{q}_4(q))) = 1 + q_4.$$
Thus $n_{\beta_4 + k \Psi_4} = 1$ when $k = 0, 1$, and zero otherwise.

We do not have closed formulas for $\exp(-f(\check{q}_1(q),\check{q}_2(q)))$ and $\exp(-g(\check{q}_1(q),\check{q}_2(q)))$, still their power series expansion can be obtained with the help of a computer program. The corresponding open GW invariants can be extracted from the power series expansions. Some open GW invariants computed this way are shown in Tables \ref{inv1} and \ref{inv2}.

\begin{table} \footnotesize
\begin{tabular}{c|r|r|r|r|r|r|r|r|c}
        \hline
	     & $k_2=0$ & $k_2=1$ & $k_2=2$ & $k_2=3$ & $k_2=4$ & $k_2=5$ & $k_2=6$ & $k_2=7$  \\
	\hline
	 $k_1=0$ & $1$		& $0$			& $0$			& $0$			& $0$			& $0$			& $0$			& $0$ \\
	\hline
	 $k_1=1$ & $-2$	& $-2$		& $0$			&	$0$			& $0$			& $0$			& $0$			& $0$ \\
	\hline
	 $k_1=2$ & $5$		& $8$			& $9$			& $8$			& $12$		& $16$		& $20$		& $24$ \\
	\hline
	 $k_1=3$ & $-32$	& $-70$		& $-96$		& $-110$	& $-140$	& $-252$	& $-504$	& $1056$	\\
	\hline
	 $k_1=4$ & $286$	& $800$		& $1323$	& $1744$	& $2268$	& $3528$ 	& $6700$	& $14120$	 \\
	\hline
	 $k_1=5$ &$-3038$&$-10374$ &$-20232$ &$-30382$ &$-42030$ & $-62838$&$-109704$&$-241020$ \\
	\hline
	 $k_1=6$ &$35870$&$144768$ &$326190$ &$552328$ &$824941$ &$1244256$&$2496039$&$5108760$ \\
	\hline
	 $k_1=7$ &$-454880$&$-2119298$&$-5424408$&$-10251170$&$-16592576$&$-30962188$&$-57926758$&$-115570212$ \\
	\hline
\end{tabular}
\caption{$n_{\beta_1 + k_1 \Psi_1 + k_2 \Psi_2}$}
\label{inv1}
\end{table}

\begin{table} \footnotesize
\begin{tabular}{c|r|r|r|r|r|r|r|r|c}
        \hline
	      & $k_2=0$ & $k_2=1$ & $k_2=2$ & $k_2=3$ & $k_2=4$ & $k_2=5$ & $k_2=6$ & $k_2=7$  \\
	\hline
	$k_1=0$ & $1$		& $1$			& $0$			& $0$			& $0$			& $0$			& $0$			& $0$			 \\
	\hline
	$k_1=1$ & $0$   & $-2$	& $-2$		& $-4$			&	$-6$			& $-8$	& $-10$		&$-12$		\\
	\hline
	$k_1=2$ & $0$ 	& $5$		& $8$			& $9$			& $20$			& $56$		& $162$		& $418$		 \\
	\hline
	$k_1=3$ & $0$ 	&$-32$	& $-70$		& $-96$		& $-140$	& $-300$	& $-768$	& $-2220$	 \\
	\hline
	$k_1=4$ & $0$		&$286$	& $800$		& $1323$	& $1936$	& $3360$	& $7280$ 	& $17910$ \\
	\hline
	$k_1=5$ & $0$   &$-3038$&$-10374$ &$-20232$ &$-32098$ &$-52630$ &$-101250$&$-172556$  \\
	\hline
	$k_1=6$ & $0$		&$35870$&$144768$ &$326190$ &$570556$ &$947505$ &$2158152$&$4976917$  \\
	\hline
	$k_1=7$ &$0$&$-454880$&$-2119298$&$-5424408$&$-10466390$&$-16175680$&$-28112692$&$-65956176$  \\
	\hline
\end{tabular}\caption{$n_{\beta_2 + k_1 \Psi_1 + k_2 \Psi_2}$}
\label{inv2}
\end{table}

\begin{remark}
Notice that all entries in the above tables are integers.  Indeed the relation between open and closed invariants given by Proposition 4.4 and Theorem 4.5 of \cite{LLW10} still applies to this case.  Let $h = \beta_1 + \beta_7 \in H_2(X)$, let $\tilde{X}$ be the blow-up of $X$ at a generic point, and $\tilde{h} \in H_2(\tilde{X})$ the proper transform of $h$ under this blow-up.  Then
$$ n_{\beta_1 + k_1 \Psi_1 + k_2 \Psi_2} = \langle 1 \rangle^{\tilde{X}}_{0,0,\tilde{h} + k_1 \Psi_1 + k_2 \Psi_2}. $$
Since the class $\tilde{h} + k_1 \Psi_1 + k_2 \Psi_2$ is primitive, its GW invariant is an integer.  Similarly $n_{\beta_2 + k_1 \Psi_1 + k_2 \Psi_2}$ is equal to closed invariants of a primitive class, and this gives a geometric reason why they are integers. See also \cite[Remark 5.7]{CLL} for a related comment.
\end{remark}

\section{The relation with Seidel representations} \label{Seidel_rep}

In \cite{G-I11}, Gonz\'alez and Iritani studied the relation between the Seidel elements $\tilde{S}_i$ \cite{seidel97}, \cite{McDuff_seidel} and the so-called Batyrev elements $\tilde{D}_i$ of a semi-Fano toric manifold $X$.  More precisely, they proved the following formula:
\begin{theorem}[Theorem 1.1 in \cite{G-I11}]
Let $X$ be a semi-Fano toric manifold. Then the Seidel element $\tilde{S}_i$ and the Batyrev element $\tilde{D}_i$ are related by
\begin{equation}\label{Seidel=Batyrev}
\tilde{S}_i(q)=\exp\left(-g_0^{(i)}(\check{q})\right)\tilde{D}_i(\check{q})
\end{equation}
under the toric mirror map $\check{q}=\check{q}(q)$.
\end{theorem}
Here, the functions $g_0^{(i)}(\check{q})$ for $i=1,\ldots,m$ are as defined in \eqref{gi}. They appear as part of the toric mirror map for the symplectic $X$-bundle $E_i\to \proj^1$ which is constructed and used to define the Seidel element $\tilde{S}_i$ \cite{seidel97}.  We refer to \cite{G-I11} for the precise definitions of $\tilde{S}_i$ and $\tilde{D}_i$.  Note that the mirror moduli coordinates $\check{q}$ are written as $y$ there.

For $i=1,\ldots,m$, we consider the following equation between the generating function $1+\delta_i(q)$ of open GW invariants and the function $g_0^{(i)}(\check{q})$:
\begin{equation}\label{opGW=g_0}
1+\delta_i(q)=\exp\left(g_0^{(i)}(\check{q})\right)
\end{equation}
under the toric mirror map $\check{q}=\check{q}(q)$.

The purpose of this section is to prove
\begin{theorem}\label{rel_Seidel}
\eqref{PF=LF} is equivalent to \eqref{opGW=g_0}.
\end{theorem}

\eqref{Seidel=Batyrev}-\eqref{opGW=g_0} suggest that there is a close relationship between Seidel representation and open GW invariants. Also, \eqref{opGW=g_0} provides a complete and effective calculation of the generating functions $\delta_i$ of open GW invariants.

We now begin with preparation of the proof of Theorem \ref{rel_Seidel}.\footnote{This proof of Theorem \ref{rel_Seidel} has already appeared in \cite{CLLT12}; we include it here just for completeness.} First of all, notice that the formula $W^{\text{PF}} = W^{\text{LF}}$ implies the multiplicative relation
\begin{equation*}
q^d\prod_{i=1}^m(1+\delta_i(q))^{D_i\cdot d}=\check{q}^d
\end{equation*}
for any $d\in H_2(X,\integer)$. By the results of Gonz\'alez-Iritani \cite{G-I11}, we know that the Seidel elements and Batyrev elements satisfy the multiplicative relations
\begin{equation*}
\prod_{i=1}^m\tilde{S}_i^{D_i\cdot d}=q^d,\quad \prod_{i=1}^m\tilde{D}_i^{D_i\cdot d}=\check{q}^d
\end{equation*}
respectively. Hence, the formula $W^{\text{PF}} = W^{\text{LF}}$ implies that
\begin{eqnarray*}
\prod_{i=1}^m(1+\delta_i(q))^{D_i\cdot d} & = & \check{q}^d/q^d \\
& = & \prod_{i=1}^m(\tilde{D}_i/\tilde{S}_i)^{D_i\cdot d} \\
& = & \prod_{i=1}^m\left(\exp\left(g_0^{(i)}(\check{q})\right)\right)^{D_i\cdot d}
\end{eqnarray*}
under the toric mirror map $\check{q}=\check{q}(q)$.

Conversely, suppose that we have
$$\prod_{i=1}^m(1+\delta_i(q))^{D_i\cdot d}=\prod_{i=1}^m\left(\exp\left(g_0^{(i)}(\check{q})\right)\right)^{D_i\cdot d}$$
for any $d\in H_2(X,\integer)$. Recall that
$$W^{\mathrm{LF}}=\sum_{i=1}^m(1+\delta_i(q))Z_i,$$
where $Z_i = z_i$ for $i = 1, \ldots, n$ and $Z_j = q^{\alpha_j} z^{v_j}$ for $j = n+1, \ldots, m$.  $Z_j$ are related by
$$\prod_{i=1}^m Z_i^{D_i\cdot d}=q^d$$
for any $d\in H_2(X,\integer)$. Using the change of variables
$$\tilde{Z}_i=(1+\delta_i(q))Z_i,$$
we can write $W^{\mathrm{LF}}=\sum_{i=1}^m\tilde{Z}_i$ where the coordinates $(\tilde{Z}_1,\ldots,\tilde{Z}_m)$ are now related by
\begin{eqnarray*}
\prod_{i=1}^m\tilde{Z}_i^{D_i\cdot d} & = & q^d\prod_{i=1}^m(1+\delta_i(q))^{D_i\cdot d} \\
& = & q^d\prod_{i=1}^m\left(\exp\left(g_0^{(i)}(\check{q})\right)\right)^{D_i\cdot d} \\
& = & q^d(\check{q}^d/q^d) \\
& = & \check{q}^d.
\end{eqnarray*}
Hence, we have $W^{\mathrm{PF}}_q=W^{\mathrm{LF}}_q$.

In summary, we have shown that \eqref{PF=LF} is equivalent to the equations
\begin{equation}\label{prodopGW=prodg_0}
\prod_{i=1}^m(1+\delta_i(q))^{D_i\cdot d}=\prod_{i=1}^m\left(\exp\left(g_0^{(i)}(\check{q})\right)\right)^{D_i\cdot d}
\end{equation}
for any $d\in H_2(X,\integer)$. Clearly, \eqref{opGW=g_0} implies \eqref{prodopGW=prodg_0}. We want to show that this is in turn equivalent to \eqref{opGW=g_0}. To proceed, observe that the proof of Corollary \ref{delta_i=0} implies the following

\begin{lemma}
If $g_0^{(i)}(\check{q})$ vanishes, then the generating function $\delta_i(q)$ also vanishes.
\end{lemma}

For $i=1,\ldots,m$, let
$$A_i(q):=\log\left(e^{-g_0^{(i)}(\check{q}(q))}(1+\delta_i(q))\right).$$
By the fact that any convex polytope with nonempty interior in $\real^n$ has at least $n+1$ vertices, at least $n+1$ of the functions $g_0^{(i)}$ are vanishing (this is Corollary 4.6 in Gonz\'alez-Iritani \cite{G-I11}). Hence, by the above lemma, at least $n+1$ of the functions $A_i(q)$ are vanishing.

Now, by taking logarithms, the equation \eqref{prodopGW=prodg_0} becomes
\begin{equation}\label{log(prodopGW=prodg_0)}
\sum_{i=1}^m (D_i\cdot d)A_i(q)=0
\end{equation}
for any $d\in H_2(X,\integer)$. There are at most $m-(n+1)=l-1$ nonzero $A_i's$ in these equations.

We are now ready to prove Theorem \ref{rel_Seidel}:
\begin{proof}[Proof of Theorem \ref{rel_Seidel}]
It suffices to prove that the system of linear equations \eqref{log(prodopGW=prodg_0)} has a unique solution (i.e. $A_i(q)\equiv0$ for all $i$). Without loss of generality, we assume that $g_0^{(1)},\ldots,g_0^{(s)}$ are the non-vanishing functions. By Proposition \ref{gi=0}, for $j=1,\ldots,s$, $v_j$ is not is vertex of the fan polytope of $X$. Let $F_j$ be the minimal face of the fan polytope of $X$ which contains $v_j$. Then $F_j$ is the convex hull of primitive generators $v_{p_1},\ldots,v_{p_k}$ which are vertices of the fan polytope of $X$. So there exist positive integers $a_1,\ldots,a_k,b$ such that $a_1v_{p_1}+\ldots+a_kv_{p_k}-bv_i=0$. This primitive relation corresponds to a class $d_j\in H_2(X,\integer)$ such that $D_j\cdot d_j=-b<0$, $D_{p_t}\cdot d_j=a_t$ and $D_r\cdot d_j=0$ when $r$ is none of $j,p_1,\ldots,p_k$. (This argument is in fact contained in the proof of Theorem 1.2 in \cite{G-I11}.)

It then follows from the system of equations:
$$\sum_{i=1}^m (D_j\cdot d)A_i(q)=0,\ j=1,\ldots,s$$
that we have $A_i(q)\equiv 0$ for all $i$. This completes the proof of Theorem \ref{rel_Seidel}.
\end{proof}

\bibliographystyle{amsplain}
\bibliography{geometry}

\end{document}